\documentclass[11pt,oneside,reqno]{article}
\usepackage[margin=1.3in]{geometry}

\usepackage{amsmath,amsthm,amssymb,color}
\usepackage[authoryear]{natbib}
\usepackage{booktabs}

\RequirePackage{amsthm,amsmath,amsfonts,amssymb,mathrsfs,dsfont,mathtools,thmtools}
\RequirePackage{natbib}
\RequirePackage[colorlinks,citecolor=blue,urlcolor=blue]{hyperref}

\usepackage[capitalize]{cleveref}
\RequirePackage{graphicx}

\crefname{equation}{}{}
\crefformat{equation}{\textup{(#2#1#3)}}
\crefrangeformat{equation}{\textup{(#3#1#4)--(#5#2#6)}}
\crefdefaultlabelformat{#2\textup{#1}#3}
\crefname{lemma}{Lemma}{Lemmas}
\crefname{page}{p.}{pp.}

\usepackage{bbm}
\usepackage{mathrsfs}
\usepackage{graphicx}
\usepackage{tikz}
\usepackage{enumitem}
\usetikzlibrary{arrows, automata}
\usetikzlibrary{calc}
\usetikzlibrary{positioning}

\numberwithin{equation}{section}
\allowdisplaybreaks[4]

\theoremstyle{plain}
\newtheorem{theorem}{Theorem}[section]
\newtheorem{proposition}{Proposition}[section]
\newtheorem{lemma}{Lemma}[section]
\newtheorem{corollary}{Corollary}[section]

\theoremstyle{definition}

\newtheorem{remark}{Remark}[section]

\makeatletter

\newcount\minute
\newcount\hour
\newcount\hourMins
\def\now{%
\minute=\time%
\hour=\time \divide \hour by 60%
\hourMins=\hour \multiply\hourMins by 60%
\advance\minute by -\hourMins%
\zeroPadTwo{\the\hour}:\zeroPadTwo{\the\minute}%
}
\def\zeroPadTwo#1{\ifnum #1<10 0\fi#1}

\renewcommand{\cite}{\citet}

\def\^#1{\ifmmode {\mathaccent"705E #1} \else {\accent94 #1} \fi}
\def\~#1{\ifmmode {\mathaccent"707E #1} \else {\accent"7E #1} \fi}

\def\*#1{#1^\ast}
\edef\-#1{\noexpand\ifmmode {\noexpand\bar{#1}} \noexpand\else \-#1\noexpand\fi}
\def\>#1{\vec{#1}}
\def\.#1{\dot{#1}}

\def\atop{\@@atop}
\def\*#1{\mathscr{#1}}

\renewcommand{\leq}{\leqslant}
\renewcommand{\geq}{\geqslant}
\newcommand{\eps}{\varepsilon}
\renewcommand{\eps}{\varepsilon}

\newcommand{\eq}{\eqref}

\newcommand{\IE}{\mathbbm{E}}
\newcommand{\IP}{P}

\newcommand{\rel}{\mathop{\mathrm{relint}}\nolimits}
\newcommand{\Cov}{\mathop{\mathrm{Cov}}}

\newcommand{\IR}{\mathbb{R}}

\def\be#1{\begin{equation*}#1\end{equation*}}
\def\ben#1{\begin{equation}#1\end{equation}}
\def\bes#1{\begin{equation*}\begin{split}#1\end{split}\end{equation*}}
\def\besn#1{\begin{equation}\begin{split}#1\end{split}\end{equation}}

\def\ba#1{\begin{align*}#1\end{align*}}
\def\ban#1{\begin{align}#1\end{align}}

\def\norm#1{\Vert#1\Vert}

\def\abs#1{\left\vert#1\right\vert}

\def\mid{\vert}

\def\beqn#1\eeqn{\begin{align}#1\end{align}}
\def\beq#1\eeq{\begin{align*}#1\end{align*}}

\usepackage{graphicx}
\usepackage{latexsym}
\usepackage{amsmath,amsthm,amssymb,amscd}
\usepackage{epsf,amsmath}

\def\E{{\IE}}
\def\P{{\IP}}

\newcommand{\ul}[1]{\underline{#1}}

\newcommand{\mcl}[1]{\mathcal{#1}}

\DeclareMathOperator{\rank}{rank}

\def\blfootnote{\xdef\@thefnmark{}\@footnotetext}

\makeatother

\begin{document}

\title{
High-dimensional Central Limit Theorems by Stein's Method in the Degenerate Case
}
\author{Xiao Fang, Yuta Koike, Song-Hao Liu and Yi-Kun Zhao}
\date{\it The Chinese University of Hong Kong, University of Tokyo, Southern University of Science and Technology and Dalian University of Technology} 
\maketitle

\noindent{\bf Abstract:} 
In the literature of high-dimensional central limit theorems, there is a gap between results for general limiting correlation matrix $\Sigma$ and the strongly non-degenerate case.
For the general case where $\Sigma$ may be degenerate, under certain light-tail conditions, when approximating a normalized sum of $n$ independent random vectors by the Gaussian distribution $N(0,\Sigma)$ in multivariate Kolmogorov distance, the best-known error 
rate has been $O(n^{-1/4})$, subject to logarithmic factors of the dimension. For the strongly non-degenerate case, that is, when the minimum eigenvalue of $\Sigma$ is bounded away from 0, the error rate can be improved to $O(n^{-1/2})$ up to a $\log n$ factor. In this paper, we show that the $O(n^{-1/2})$ rate up to a $\log n$ factor can still be achieved in the degenerate case, provided that the minimum eigenvalue of the limiting correlation matrix of any three components is bounded away from 0.
We prove our main results using Stein's method in conjunction with previously unexplored inequalities for the integral of the first three derivatives of the standard Gaussian density over convex polytopes. These inequalities were previously known only for hyperrectangles. Our proof demonstrates the connection between the three-components condition and the third moment Berry--Esseen bound.




\medskip



\section{Introduction and main results}

High-dimensional central limit theorem (CLT) has attracted much attention both in theory and applications since the work of \cite{chernozhukov2013gaussian,chernozhukov2017central}. 
They proved that, under mild regularity assumptions, Gaussian approximation in multivariate Kolmogorov distance is valid even when the dimension $d$ is much larger than the sample size $n$. Such results are crucial in high-dimensional statistical inferences (\cite{belloni2018high}).

The currently best known general high-dimensionl CLT result is as follows. 
Let $X_1,\dots,X_n$ be independent mean-zero random vectors in $\mathbb R^d$, $d\geq 3$. Set $W=n^{-1/2}\sum_{i=1}^nX_i$.
Suppose $W$ has covariance matrix $\Sigma=(\Sigma_{jk})_{1\leq j, k\leq d}$ with $\ul\sigma^2:=\min_{1\leq j\leq d}\Sigma_{jj}>0$. Suppose also that there exists a constant $B>0$ such that $\E\exp(|X_{ij}|/B)\leq2$ and $n^{-1}\sum_{i=1}^n\E X_{ij}^4\leq B^2$ for all $1\leq i\leq n$, $1\leq j\leq d$, where $X_{ij}$ is the $j$-th component of $X_i$.
Then, according to Theorem 2.1 in \cite{chernozhuokov2022improved}, we have
\ben{\label{eq:cckk}
\rho(W,G):=\sup_{a, b\in \IR^d}|P(a\leq W\leq b)-P(a\leq G\leq b)| \leq \frac{cB^{1/2}}{n^{1/4}}(\log(dn))^{5/4},
}
where $G\sim N(0,\Sigma)$, 
$``\leq "$ denotes the component-wise inequality for vectors,
and $c$ is a constant depending only on $\ul\sigma^2$. 

Assuming that $\Sigma$ has unit diagonal entries (this is a mild condition because $\rho(W,G)$ is invariant under component-wise scaling) and that its smallest eigenvalue, $\sigma_*^2$, is bounded away from 0 (referred to as the strongly non-degenerate case),  a separate line of works by \cite{fang2021high,lopes2022central,kuchibhotla2020high,chernozhukov2020nearly} obtained faster convergence rates in the high-dimensional CLT.
For example, according to Corollary 2.1 of \cite{chernozhukov2020nearly}, if $|X_{ij}|\leq B$ for $1\leq i\leq n, 1\leq j\leq d$, then
\ben{\label{eq:cckoike}
\rho(W,G)\leq \frac{CB}{\sigma_*^2\sqrt{n}}(\log d)^{3/2}\log n,
}
where $C$ is a universal constant. This error rate is optimal up to the $\log n$ factor.
See \cite{lopes2020bootstrapping}, \cite{das2021central}, and \cite{fang2022sharp} for optimal rates in some other special situations, namely, the case that $\Sigma$ has decreasing and vanishing diagonal entries, the case of independent coordinates, and the case of sums of log-concave random vectors.

Comparing \eq{eq:cckk} and \eq{eq:cckoike}, we see a gap between the error rates $n^{-1/4}$ and $n^{-1/2}$. It has been an open problem  whether it is possible to improve the rate in the degenerate case. We take a first step in filling this gap by proving the following result:

\begin{theorem}\label{thm:fklz}
Let $n, d\geq 3$ be integers.
Let $X_1,\dots,X_n$ be independent mean-zero random vectors in $\mathbb R^d$. Let $W=n^{-1/2}\sum_{i=1}^nX_i$ and let $\Cov(W)$ be its covariance matrix. 
Let $\Sigma$ be a $d\times d$ covariance matrix such that for all distinct $1\leq j,k,l\leq d$:
\be{
\Sigma_{jj}=1, \quad \text{(without loss of generality)}
}
\ben{\label{eq:three}
\det(\Sigma^{j,k})\geq \alpha^2>0,\ \frac{\det(\Sigma^{j,k,l})}{\det(\Sigma^{j,k})}\geq \beta^2>0,\quad  \text{(any three components are non-degenerate)}
}
where $\Sigma^{j,k}$ is the $2\times 2$ sub-matrix of $\Sigma$ by intersecting the $j$th and $k$th rows with the $j$th and $k$th columns and $\Sigma^{j,k,l}$ the $3\times 3$ sub-matrix defined analogously in the obvious way.
Let $G\sim N(0, \Sigma)$.
Suppose that there exists a constant $B>0$ such that $\E\exp(|X_{ij}|/B)\le 2$ for all $1\leq i\leq n$, $1\leq j\leq d$.
Then we have
\besn{\label{eq:fklz}
\rho(W,G)\leq& \frac{C}{\alpha^2} \norm{\Cov(W)-\Sigma}_{\infty} (\log d)\log n+\frac{CB^3}{\alpha^2 \beta^2 \sqrt{ n}} (\log d)^{2}\sqrt{\log(dn)}(\log n)^4,
}
where $\norm{\cdot}_{\infty}$ denotes the infinity norm and $C$ is an absolute constant. 

\end{theorem}

\begin{remark}
(a) Typically $\Sigma=\Cov(W)$ (we allow them to be different to facilitate applications to bootstrap approximations below) and
\cref{thm:fklz} provides a first high-dimensional CLT with optimal rate $O(n^{-1/2})$ (subject to logarithmic factors) for a general degenerate case.  
The condition $\Sigma_{jj}=1$ is only for simplicity of presentation and is not essential. By a coordinate-wise scaling, \cref{thm:fklz} translates into a result for the general case. This applies to \cref{prop:gauss} and \cref{thm:bootstrap} as well. 
Note that this operation affects the values of $B,\alpha$ and $\beta$. In particular, this condition essentially excludes the presence of nearly degenerate coordinates. Indeed, it is known that $\rho(W,G)$ could be bounded away from zero in such a situation; see Example 1 in \cite{Se86} for instance. 

(b) It remains an open problem to obtain the optimal dependence of the error bound on $\alpha, \beta, B$, and $\log d$. 
In fact, the appearance of $\alpha$ in the denominator is necessary (see \cref{rem:alpha}). Also, the appearance of $\beta$ in the denominator seems unavoidable using our method. Indeed, we will see in the proof (cf. \eq{eq:eta}) that we need to control the projection of an $\IR^d$ vector $\tilde X_i$ in certain other directions
besides those directions generating $X_{i1},\dots, X_{id}$.


(c) 
If the smallest eigenvalue $\sigma_*^2$ of $\Sigma$ is bounded away from 0, then \eq{eq:three} holds. Therefore, the condition of \cref{thm:fklz} is weaker than a strong non-degeneracy condition. See \cref{rem:sta} for statistical implications.
Note, however, the dependence of the upper bound on $B$ and $d$ is 
worse than Theorem 2.2 of \cite{chernozhukov2020nearly} for the strongly non-degenerate case.
To improve the result, we need refined techniques such as the exchangeable pair approach in Stein's method (which needs stronger conditions on the covariance matrix) and
the local stability result (cf. Eq.(9.1) of \cite{chernozhukov2020nearly}). 
It is unclear whether the stability result is still valid in the degenerate case because its proof relies on the product structure of hyperrectangles.

(d) To prove \cref{thm:fklz}, we follow the strategy of \cite{fang2021high} and \cite{lopes2022central}, which was thought to be applicable only to the strongly non-degenerate case. 
Besides an initial transformation of the problem to standard Gaussian approximation on convex polytopes (cf. \cref{sec:4.1}),
what's crucial in our new proof is a generalization of an inequality for the integral of derivatives of standard Gaussian density (see \cite{bentkus1990smooth,anderson1998edgeworth}) from hyperrectangles to polytopes (cf. \cref{lem:AHT}). The original inequality and its proof rely heavily on the product formula for the probability that a standard Gaussian vector takes values in a hyperrectangle.  For polytopes, we use the divergence theorem and associate relevant quantities to outer areas of polytopes, which generalizes the arguments in establishing Nazarov's inequality (cf. \cite{chernozhukov2017detailed}).
\end{remark}

\begin{remark}[Statistical implication]\label{rem:sta}
    It is worth mentioning that conditions in \eqref{eq:three} are empirically testable even for high-dimensional data with $d\gg n$ because we can consistently estimate the population matrix $\Sigma$ with respect to the infinity norm as long as $B=O(1)$ and $(\log d)^2=o(n)$; see the proof of \cref{thm:bootstrap}. 
    By contrast, without imposing an extra structural assumption on $\Sigma$, we cannot estimate $\Sigma$ with respect to the spectral norm accurately when $d\gg n$ (cf.~Theorem B.1 in \cite{LeLe18} and its proof). Therefore, it is generally impossible to empirically assess the strong non-degeneracy condition for high-dimensional data with $d\gg n$.
    This shows an advantage of our new result in statistical inference.
\end{remark}

Our method also works for dependent cases, for example, for sums of $m$-dependent random vectors.
Here, to facilitate applications to bootstrap approximations, we state a result assuming the existence of Stein kernel (cf. \cite{LeNoPe15}). 
We state the result as a $d$-dimensional standard Gaussian approximation on convex polytopes with $d$ facets and will explain how to transform the result into general Gaussian approximations in Kolmogorov distance in the proof of \cref{prop:gauss}.

In this paper, we use $\cdot$ to denote the inner product, and $\nabla$ the gradient. For an $r$-times differentiable function $f:\IR^d\to \IR$, we denote by $\nabla^r f(x)$ the $r$th derivative of $f$ at $x$. We use $\langle , \rangle$ to denote the dot product of tensors (sum of products of corresponding entries).

\begin{theorem}[Stein kernel bound] \label{thm:kernel}
Let $W\in \IR^d$, $d\geq 3,$ be a random vector with $\E W=0$. Suppose $W$ has a Stein kernel $\tau^W(\cdot): \IR^d\to \IR^{d\times d}$, that is,
\ben{\label{eq:steinkernel}
\E [W\cdot \nabla f(W)]=\E\langle \tau^W(W), \nabla^2 f(W)\rangle
}
for every smooth function $f: \IR^d\to \IR$ such that the expectations exist. 
Let $v_1,\dots, v_d\in \IR^d$ be distinct unit vectors, $b_1,\dots, b_d$ real numbers, and let
\ben{\label{eq:polytope}
A=\{x\in \IR^d:\  x\cdot v_j\leq b_j\ \text{for}\ j=1,\dots, d\}
}
be a convex polytope\footnote{Throughout this article, we allow convex polytopes to be unbounded.}. 
Then we have
\ben{\label{eq:kernelbound}
|P(W\in A)-P(Z\in A)|\leq C \Delta (\log d) (|\log \Delta|\vee 1),
}
where $C$ is an absolute constant, $Z\sim N(0, I_d)$,
\be{
\Delta:=\E \left[\max_{j,k} (|v_j^\top M v_j|+|v_j^\top M v_{jk}|+ |v_{jk}^\top M v_{jk}| )\right] , \quad M:=\tau^W(W)-I_d,
}
and the max is taken over all $1\leq j\ne k\leq d$ such that $v_k\ne - v_j$ and $v_{jk}\in \IR^d$ is the unit vector in the two-dimensional subspace spanned by $\{v_j, v_k\}$ with $v_{jk}\cdot v_j=0$ and $v_{jk}\cdot v_k>0$.
\end{theorem}

In the special case that $A$ is a hyperrectangle, we have $v_{jk}=v_{k}$(if $v_k\ne -v_j$) and \cref{thm:kernel} gives the same result as Theorem 1.1 of \cite{fang2021high}.
Applying \cref{thm:kernel} to (possibly degenerate) Gaussian-to-Gaussian comparison, we obtain the following result.

\begin{proposition}[Gaussian-to-Gaussian Comparison]\label{prop:gauss}
Let $Z_1$ and $Z_2$ be centered Gaussian random vectors in $\IR^d$, $d\geq 3$, with covariance matrices $\Sigma^{(1)}$ and $\Sigma$, respectively.
Suppose $\Sigma_{jj}=1$ and $1-\Sigma_{jk}^2\geq \alpha^2>0$ for all $1\leq j\ne k\leq d$.
Then we have
\ben{\label{eq:gauss}
\rho(Z_1, Z_2):=\sup_{a, b\in \IR^d}\left|P(a\leq Z_1\leq b)-P(a\leq Z_2\leq b)  \right|\leq \frac{C}{\alpha^2}\Delta (\log d) (|\log \Delta|\vee 1),
}
where $C$ is an absolute constant and $\Delta=\max_{1\leq j,k\leq d}|\Sigma^{(1)}_{jk}-\Sigma_{jk}|$. 
\end{proposition}

\begin{remark}\label{rem:alpha}
In Proposition 2.1 of \cite{chernozhuokov2022improved}, it was shown that 
the left-hand side of \eq{eq:gauss} is bounded by $C\sqrt{\Delta} (\log d)$. Therefore, \cref{prop:gauss} presents a significant improvement, provided that $\alpha^2$ is bounded away from 0. 
We remark that Proposition B.1 of \cite{chernozhuokov2022improved} shows that there exists an example where $C(\Delta/\alpha)(\log d)$ gives a lower bound for $\rho(Z_1,Z_2)$. So the non-degeneracy condition on $\alpha^2$ is necessary to achieve the improved error rate while it is unclear whether the dependence on $\alpha$ of our bound is optimal or not.
\end{remark}

\cref{prop:gauss} will help us to obtain $O(n^{-1/2})$ error rates in bootstrap approximations in degenerate cases, provided that $\alpha^2$ is bounded away from 0.
A motivation for bootstrap approximation is that the covariance matrix $\Sigma$ is in practice typically unknown and we need to estimate $P(a\leq G\leq b)$ in \eq{eq:fklz} from the sample $X_1,\dots, X_n$.
We refer to \cite{chernozhukov2023high} for relevant techniques and known results of bootstrap approximations in high-dimensions.
Here, we only consider Gaussian multiplier bootstrap and similar results can be obtained for general multiplier bootstraps as well as for the empirical bootstrap.

Let $\xi_1,\dots, \xi_n$ be i.i.d.\ $N(0,1)$ random variables that are independent of $X:=(X_1,\dots, X_n)$. Denote $\bar X:=(\bar X_1,\dots, \bar X_d)^\top=n^{-1}\sum_{i=1}^n  X_i$ and consider the Gaussian multiplier bootstrap version of $W$:
\be{
W^\xi:=\frac{1}{\sqrt{n}}\sum_{i=1}^n\xi_i(X_i-\bar X).
}
Let $\Sigma=\Cov(W)$.
We are interested in bounding
\be{
\rho^\xi:=\sup_{a, b\in \IR^d} |P(a\leq W^\xi\leq b| X)-P(a\leq G\leq b)|,\quad G\sim N(0,\Sigma).
}
Following the proof of Corollary 3.1 of \cite{chernozhukov2020nearly} and using our \cref{prop:gauss}, we obtain the following result (see \cref{sec:pr-boot} for details).

\begin{theorem}[Gaussian multiplier bootstrap approximation]\label{thm:bootstrap}
Under the conditions of \cref{thm:fklz}, allow $\beta^2$ to be 0 and suppose $\Sigma=\Cov(W)$. 
Then there exists a universal constant $C>0$ such that, for any $\gamma\in(0,1)$, we have with probability at least $1-\gamma$,
\be{
\rho^\xi\leq \frac{CB^2 (\log(d/\gamma))^{3/2} \log n}{\alpha^2  \sqrt{n}}.
}
\end{theorem}

\cref{thm:bootstrap}, combined with \cref{thm:fklz} (choosing $\Sigma=\Cov(W)$), gives explicit error bounds in approximating $P(W\leq b)$ by the practical quantity $P(W^\xi\leq b\mid X)$.



\paragraph{Organization.} The paper is organized as follows. Key lemmas are given in \cref{sec:lem}. Their proofs are deferred to \cref{sec:prooflem}. We prove Theorems \ref{thm:kernel}--\ref{thm:bootstrap} and \cref{prop:gauss} in \cref{sec:proof1} and \cref{thm:fklz} in \cref{sec:proof2}. Some standard analysis are contained in \cref{sec:app}.

\paragraph{Notation.} Besides the notation introduced above \cref{thm:kernel}, we use $\otimes$ to denote tensor product, $\phi_d(\cdot)$ the $d$-dimensional standard normal density function. We use $C$ to denote positive, absolute constants, which may defer from line to line.

\section{Lemmas}\label{sec:lem}


We first state three lemmas that are needed in the proofs of \cref{thm:fklz,thm:kernel}. 
In the remaining of this section, let 
\be{
A=\{x\in \IR^d:\  x\cdot v_j\leq b_j\ \text{for}\ j=1,\dots, d\}
}
be a convex polytope, where $v_1,\dots, v_d$ are distinct unit vectors in $\IR^d$, $d\geq 3$, and $b_1,\dots, b_d$ are real numbers. Let 
\be{
V_1=\{v_1,\dots, v_d\}.
}

The following two results are equivalent forms of Nazarov's inequality (see \cite{chernozhukov2017detailed}).
\begin{lemma}\label{cor1}
Let $Z$ be a standard normal random vector in $\mathbb{R}^d$, $d\geq 3$. Then we have 
\begin{equation}
    P(Z \cdot v_j\leq b_j+\epsilon, 1\leq j\leq d)-P(Z \cdot v_j\leq b_j, 1\leq j\leq d)\leq \epsilon \left(\sqrt{2\log d}+2\right).
\end{equation}
\end{lemma}


\begin{corollary}[Gaussian anti-concentration inequality]\label{l1}
Let $G$ be a centered Gaussian vector in $\mathbb{R}^d$, $d\geq 3$, such that $\min_{1\leq j\leq d}\E G_j^2\geq\underline{\sigma}^2$ for some $\underline{\sigma}>0$. Then, for any $b\in\mathbb{R}^d$ and $\epsilon>0$,
\[
P(G\leq b+\epsilon \textbf{1}_d)-P(G\leq b)\leq\frac{\epsilon}{\underline{\sigma}}\left(\sqrt{2\log d}+2\right),
\]
where $\{G\leq b\}:=\{G_j\leq b_j: 1\leq j\leq d\}$ and $\textbf{1}_d$ denotes the $\IR^d$-vector with all components equal to 1.
\end{corollary}

For each pair of unit vectors $v_j, v_k\in V_1$ such that $v_k\ne - v_j$, we define $v_{jk}\in \IR^d$ to be the unit vector in the two-dimensional subspace spanned by $\{v_j,v_k\}$ with $v_{jk}\cdot v_j=0$ and $v_{jk}\cdot v_k>0$ as in \cref{thm:kernel}. We define $v_{kj}$ similarly by requiring $v_{kj}\cdot v_k=0$ and $v_{kj}\cdot v_j>0$.
Let $V_2$ be the collection of such $v_{jk}$'s and $v_{kj}$'s. Note that $|V_2|=O(d^2)$, where $|\cdot|$ denotes the cardinality when applied to a set.

For each triple of unit vectors $v_j, v_k, v_l\in V_1$ such that the span of them is three-dimensional, we define $v_{jkl}=v_{kjl}\in \IR^d$ to be the unit vector in the three-dimensional subspace spanned by $\{v_j, v_k, v_l\}$ with $v_{jkl}\cdot v_j=0$, $v_{jkl}\cdot v_k=0$ and $v_{jkl}\cdot v_l>0$. We define $v_{jlk}, v_{ljk}$ and $v_{klj}, v_{lkj}$ analogously. Let $V_3$ be the collection of such $v_{jkl}$'s. Note that $|V_3|=O(d^3)$.

\begin{lemma}\label{lem:AHT}
Let $A$ be a convex polytope with corresponding sets of unit vectors $V_1, V_2, V_3$ as above.
Recall $\phi_d(z)$ is the $d$-dimensional standard normal density. Assume $d\geq 3$.
We have
\ben{\label{eq:AHTdeg1}
\left|\int_A (u\cdot \nabla \phi_d(z)) dz \right| \leq C(\log d)^{1/2} \max_{v\in V_1} |u\cdot v|,\quad   \forall \ u\in \IR^d
}
and for any $d\times d$ matrix $M$,
\ben{\label{eq:AHTdeg2}
\left|\int_A\langle M, \nabla^2 \phi_d (z) \rangle dz \right|\leq C(\log d) \max_{j,k} (|v_j^\top M v_j|+|v_j^\top M v_{jk}|+ |v_{jk}^\top M v_{jk}| ),
}
the max is taken over all $1\leq j\ne k\leq d$ such that $v_k\ne - v_j$.
Moreover, if the minimum angle between $v_j$ and $v_k$ for $j\ne k$ is $\alpha>0$ and the minimum angle between $v_l$ and the two-dimensional subspace spanned by $\{v_j, v_k\}$ for any triple of distinct indices $1\leq j,k,l\leq d$ with $v_k\ne -v_j$ is $\beta>0$,
then for any $d\times d\times d$ tensor $T$:
\ben{\label{eq:AHTdeg3}
\left|\int_A\langle T, \nabla^3 \phi_d (z) \rangle dz \right|\leq \frac{C(\log d)^{3/2}}{\alpha \beta}  \max_{v_1\in V_1, v_2 \in V_1\cup V_2,\atop v_3\in V_1 \cup V_2\cup V_3} |\langle T, v_1\otimes v_2\otimes v_3\rangle| .
}




\end{lemma}

\begin{remark}
Higher-order derivative bounds may be obtained following the proof of \cref{lem:AHT}. Because they have more complicated expressions and require stronger conditions, we do not consider them in this paper. Moreover, a general result for convex polytopes in $\IR^{d_2}$ with $d_1$ facets easily follows from the case $d_1=d_2=d$. In fact, if $d_2>d_1$, we can restrict the consideration to a $d_1$-dimensional subspace, and if $d_2<d_1$, we can extend the $d_2$-dimensional standard normal to $\IR^{d_1}$ by concatenating it with independent $N(0,1)$'s. The final bound will only depend on $d_1$, the number of facets of the convex polytope.
\end{remark}

\begin{remark}
Similar upper bounds as in \cref{lem:AHT}, but for quantities with the absolute value inside of the dot product and $A$ replaced by hyperrectangles, were known in the literature (cf. \cite{bentkus1990smooth,anderson1998edgeworth}) and proved using the product form of hyperrectangles. We realized that when the absolute value is outside of the integral, we can use the divergence theorem to obtain the new bounds and they are enough to prove high-dimensional CLTs. This observation was the starting point of this work. 
\end{remark}


The aim of the next lemma is to prove a polytope analog of Lemma 10.5 in \cite{lopes2022central} (see also Lemma 6.3 in \cite{chernozhukov2020nearly}). It will be used for the recursion argument below \eq{eq:eta}.
Let $A$ be as above. For any $t\in\mathbb R$, define
\ben{\label{at}
A(t):=\{x\in\mathbb R^d:x\cdot v_j\leq b_j+t\text{ for }j=1,\dots,d\}.
}
\begin{lemma}\label{lem:vanish}
For any $\kappa>0$ and $u, u_1, u_2, u_3\in\mathbb R^d$, we have
\ben{\label{eq:vanish}
\sup_{x\notin A(\kappa)\setminus A(-\kappa)}\abs{\int_A\langle u,\nabla\phi_d(z-x)\rangle dz}\leq d \phi_1(\kappa)\max_{v_1\in V_1}|u\cdot v_1|.
}
\ben{\label{eq:vanish3}
\sup_{x\notin A(\kappa)\setminus A(-\kappa)}\abs{\int_A\langle u_1\otimes u_2\otimes u_3,\nabla^3\phi_d(z-x)\rangle dz}\leq C d^3 e^{-\kappa^2/4}\max_{v_1\in V_1,v_2\in V_1 \cup V_2,\atop   v_3\in V_1\cup V_2\cup V_3} (|u_1\cdot v_1| |u_2\cdot v_2|||u_3\cdot v_3|).
}
\end{lemma}
\cref{lem:vanish} also generalizes to higher-order derivatives in an obvious way. We do not pursue the most general form because the third-derivative bound \eq{eq:vanish3} suffices for our purpose.


\section{Proof of Theorems \ref{thm:kernel}--\ref{thm:bootstrap} and Proposition \ref{prop:gauss}}\label{sec:proof1}

In this section, we first prove \cref{thm:kernel}, which is simpler than that of \cref{thm:fklz} and shows the basic ideas. We then use \cref{thm:kernel} to prove \cref{prop:gauss} and \cref{thm:bootstrap}.

\subsection{Proof of Theorem \ref{thm:kernel}}

Without loss of generality, in proving \eq{eq:kernelbound}, we assume $W$ and $Z$ are independent.
For fixed unit vectors $v_1,\dots, v_d\in \IR^d$ as in the condition of \cref{thm:kernel},
let $\mcl K_d(v):=\mcl K_d(v_1,\dots, v_d)$ be the collection 
of convex polytopes of form
\be{
\{x\in \IR^d:\  x\cdot v_j\leq b_j\ \text{for}\ j=1,\dots, d\},
}
where $b_j$, $1\leq j\leq d$, are real numbers.

Modifying the argument leading to Eq.(2.6) of \cite{fang2021high} (simply changing the $\epsilon$-rectangle $R(0;\epsilon)$ in their Lemma 2.4 to the $\epsilon$-polytope $\{x\in \IR^d: x\cdot v_j\leq \epsilon\ \text{for}\ j=1,\dots,  d\}$), we obtain 
\ben{\label{3}
|P(W\in A)-P(Z\in A)|\leq  C\left(\sup_{h=1_A: A\in \mcl{K}_d(v)} |\E T_t \tilde h(W)| + e^t \sqrt{t} \log d   \right),\quad \forall\ t>0,
}
where 
\be{
T_t \tilde h(x)=\E h(e^{-t}x+\sqrt{1-e^{-2t}}Z)-\E h(Z), \quad Z\sim N(0,
I_d).
}




Now we fix $A\in\mcl{K}_d(v)$, write $h:=1_A$ and proceed to bound $|\E T_t\tilde h(W)|$. 
Similar to (1.14) and (3.1) of \cite{BhHo10}, one can verify that 
\ben{\label{psi}
\psi_t(w)=-\int_t^\infty T_s\tilde h(w)ds
}
is a solution to the Stein equation
\ben{\label{eq:stein}
\langle I_d,\nabla^2\psi_t(w)\rangle-w \cdot \nabla \psi_t(w)=T_t \tilde h(w).
}
Thus we have
\begin{equation}\label{n03}
\E T_t\tilde h(W)=\E[\langle I_d, \nabla^2\psi_t(W)\rangle-W\cdot \nabla \psi_t(W)].
\end{equation}
We can rewrite $T_s\tilde h(w)$ as
\ben{\label{Tsh}
T_s\tilde h(w)= h_s(e^{-s}w)-\E h(Z),
}
where $h_s:=\mcl{N}_{\sqrt{1-e^{-2s}}} h$ and $\mcl N_\sigma h(x):=\E h(x+\sigma Z)$.
In particular, $\psi_t$ is infinitely differentiable and
\ben{\label{psi-deriv}
\partial_{j_1,\dots, j_r}\psi_t(w)=-\int_t^\infty e^{-rs}[\partial_{j_1,\dots, j_r} h_s(e^{-s}w)]ds.
}
To bound $|\E T_t\tilde h(W)|$, we use the definition of Stein kernel in \eq{eq:steinkernel} and integration by parts to obtain
\bes{
\E T_t\tilde h(W)=&\E[\langle I_d, \nabla^2\psi_t(W)\rangle-W\cdot \nabla \psi_t(W)]\\
=&\E\langle I_d-\tau^W(W), \nabla^2 \psi_t(W)\rangle\\
=&\sum_{j,k=1}^d \int_t^\infty e^{-2s} \E(\tau_{jk}^W(W)-\delta_{jk}) \int_{\IR^d} \partial_{jk} h(e^{-s}W+\sqrt{1-e^{-2s}}z)\phi_d(z) dz ds\\
=& \int_t^\infty (\frac{e^{-s}}{\sqrt{1-e^{-2s}}})^2 \E\left[ \int_{\frac{A-e^{-s}W}{\sqrt{1-e^{-2s}}}}\langle \tau^W(W)-I_d, \nabla^2 \phi_d(z)\rangle dz\right]ds,
}
where $\delta_{jk}$ denotes the Kronecker delta, and $(A-a)/b$ denotes the set of all the elements in $A$ after the linear transformation.
Using \eq{eq:AHTdeg2}, we obtain
\ben{\label{eq:kernelproof1}
|\E T_t\tilde h(W)|\leq C \Delta (\log d) (|\log t|\vee 1).
}
In proving \eq{eq:kernelbound}, we can assume without loss of generality that $\Delta\leq 1$; otherwise, \eq{eq:kernelbound} holds with $C=1$.
Choosing $t=\Delta^2$ and using \eq{3} finish the proof.

\subsection{Proof of Proposition \ref{prop:gauss}}


Recall that the diagonal entries of $\Sigma$ are all equal to 1.
Without loss of generality, we assume $\Sigma$ is full rank. If not, the result follows by perturbing the off-diagonal entries of $\Sigma$ by arbitrarily small amounts (for example, by considering $(G+\eps Z')/\sqrt{1+\eps^2}$ where $Z'\sim N(0, I_d)$ and independent of $G$) and using an approximation argument. 

Write 
\be{
\Sigma=V^{-1} (V^{-1})^\top, \quad (V^{-1})^\top =(v_1,\dots, v_d), \quad |v_j|=1\ \text{for all}\ 1\leq j\leq d.
}
Here $v_1,\dots, v_d\in \IR^d$ will be the unit vectors appearing in \cref{thm:kernel}. Let $Z\sim N(0,I_d)$ and let $\tilde Z_1:=V Z_1$. Then the problem transforms to bounding 
\be{
|P(\tilde Z_1\in A)-P(Z\in A)|,
}
where
\be{
A=\{x\in \IR^d:\  a_j\leq x\cdot v_j\leq b_j\ \text{for}\ j=1,\dots, d\}.
}
We first assume $a_j=-\infty$ for all $j$ so $A$ becomes \eq{eq:polytope}. From Gaussian integration by parts, $\tilde Z_1$ has Stein kernel $V\Sigma^{(1)} V^\top$ and $Z$ has Stein kernel $I_d$. Next, we check the value of $\Delta$ appearing in \cref{thm:kernel} in this case. 
Denote $\Delta_{jk}:=\Sigma^{(1)}_{jk}-\Sigma_{jk}$.
We have
\be{
v_j^\top (V \Sigma^{(1)} V^\top-I_d) v_j=\Sigma_{jj}^{(1)}-1
=\Delta_{jj},
}
and for $k\ne j$ (note that $v_k\ne - v_j$ by the full rank assumption at the beginning of this proof),
\be{
v_{jk}=\frac{v_k-(v_j\cdot v_k)v_j}{|v_k-(v_j\cdot v_k)v_j|},\quad v_j\cdot v_k=\Sigma_{jk},\quad |v_k-(v_j\cdot v_k)v_j|^2=1-\Sigma^2_{jk},
}
\be{
v_j^\top (V \Sigma^{(1)} V^\top-I_d) v_{jk}=\frac{\Delta_{jk}-\Delta_{jj} \Sigma_{jk}}{\sqrt{1-\Sigma_{jk}^2}},
}
\ben{\label{eq:eps}
v_{jk}^\top (V \Sigma^{(1)} V^\top-I_d) v_{jk}=\frac{1}{1-\Sigma^2_{jk}} [\Delta_{kk}+\Sigma^2_{jk} \Delta_{jj}-2\Sigma_{jk}\Delta_{jk}].
}
Then, \cref{prop:gauss} follows from \cref{thm:kernel}.

If $a_j$'s are not $-\infty$, the result follows by applying \cref{thm:kernel}\footnote{Strictly speaking, we apply \cref{thm:kernel} in $\IR^{2d}$, concatenate $\tilde Z_1$ and $Z$ with an independent $\IR^d$-vector $Z'\sim N(0,I_d)$ and concatenate $\pm v_j$'s with a 0 vector in $\IR^d$.} to 
\be{
A=\{x\in \IR^d:\  x\cdot v_j\leq b_j,\ x\cdot (-v_j)\leq -a_j\ \text{for}\ j=1,\dots, d\}.
}
Recall the max in \cref{thm:kernel} is taken over all $v_k\ne -v_j$.

\subsection{Proof of Theorem \ref{thm:bootstrap}}\label{sec:pr-boot}

Applying \cref{prop:gauss} conditional on $X$, we have
\ben{\label{apply:gcomp}
\rho^\xi\leq\frac{C}{\alpha^2}\Delta_n^* (\log d) (|\log \Delta_n^*|\vee 1),
}
where $\Delta_n^*:=\|\Sigma_n-\Sigma\|_\infty$ with $\Sigma_n:=n^{-1}\sum_{i=1}^n(X_i-\bar X)(X_i-\bar X)^\top$. 
Consider the following decomposition:
\[
\Sigma_n-\Sigma=\frac{1}{n}\sum_{i=1}^n(X_iX_i^\top-\E[X_iX_i^\top])-\bar X\bar X^\top.
\]
Using Lemma 2.1 of \cite{FaKo22} with $\alpha=1$ and $\alpha=1/2$ (in their notation) respectively, we have, for any $p\geq2$,
\[
\max_j\|\bar X_j\|_p\leq C\frac{B}{n}(\sqrt{pn}+p)
\]
and
\[
\max_{j,k}\left\|\frac{1}{n}\sum_{i=1}^n(X_{ij}X_{ik}-\E[X_{ij}X_{ik}])\right\|_p
\leq C\frac{B^2}{n}(\sqrt{pn}+p^2),
\]
where $\|\cdot\|_p$ denotes the $L^p$-norm with respect to the underlying probability measure. 
Hence $\|\bar X_j \bar X_k\|_p\leq CB^2(pn+p^2)/n^2\leq CB^2p^2/n$ and
\[
\max_{j,k}\|\Sigma_{n,jk}-\Sigma_{jk}\|_p\leq C_1\frac{B^2}{n}(\sqrt{pn}+p^2),
\]
where $C_1>0$ is a universal constant. 
Thus, choosing $p=2\log(d/\gamma)$ and using the union bound and Markov's inequality, we obtain
\ba{
P\left(\Delta_n^*>eC_1\frac{B^2}{n}(\sqrt{pn}+p^2)\right)
&\leq d^2\left(eC_1\frac{B^2}{n}(\sqrt{pn}+p^2)\right)^{-p}\max_{j,k}\|\Sigma_{n,jk}-\Sigma_{jk}\|_p^p\\
&\leq d^2e^{-p}=\gamma^2\leq\gamma.
}
Consequently, there exists a universal constant $C_2\geq1$ such that with probability at least $1-\gamma$,
\[
\Delta_n^*\leq C_2\left(B^2\sqrt{\frac{\log(d/\gamma)}{n}}+B^2\frac{(\log(d/\gamma))^2}{n}\right)=:\delta_n.
\]
Now suppose that $\delta_n\log(d/\gamma)\leq1/2$. Then, from $2\geq \E \exp(|X_{ij}|/B)\geq 1+\E[X_{ij}^2]/(2B^2)$, we have 
\ben{\label{eq:star3}
2B^2\geq\max_{i,j}\E[X_{ij}^2]\geq\max_j\Sigma_{jj}=1
}
and
\[
\frac{(\log(d/\gamma))^2}{n}\leq 2B^2\sqrt{\frac{(\log(d/\gamma))^3}{n}}\sqrt{\frac{\log(d/\gamma)}{n}}
\leq\sqrt{\frac{\log(d/\gamma)}{n}}. 
\]
Hence, with probability at least $1-\gamma$,
\[
\Delta_n^*\leq 2C_2B^2\sqrt{\frac{\log(d/\gamma)}{n}}.
\]
Therefore, noting that the function $x\mapsto x(|\log x|\vee1)$ is increasing for $x\geq 0$, we have with probability at least $1-\gamma$
\ba{
\rho^\xi\leq \frac{CB^2(\log(d/\gamma))^{3/2}}{\alpha^2\sqrt n}\left(\left|\log\left(2C_2B^2\sqrt{\frac{\log(d/\gamma)}{n}}\right)\right|\vee1\right).
}
Since $2C_2B^2\sqrt{\log(d/\gamma)/n}\leq2\delta_n\leq1$, we have (recall $2C_2B^2\sqrt{\log(d/\gamma)}\geq 1$)
\ba{
\left|\log\left(2C_2B^2\sqrt{\frac{\log(d/\gamma)}{n}}\right)\right|
=-\log\left(2C_2B^2\sqrt{\frac{\log(d/\gamma)}{n}}\right)\leq\frac{1}{2}\log n.
}
Consequently, with probability at least $1-\gamma$,
\[
\rho^\xi\leq \frac{CB^2(\log(d/\gamma))^{3/2}}{\alpha^2\sqrt n}\log n.
\]
It remains to prove the claim when $\delta_n\log(d/\gamma)>1/2$. In this case, there exists a universal constant $c>0$ such that $\frac{B^2(\log(d/\gamma))^{3/2}}{\sqrt n}\geq c$; hence the desired result holds for $C\geq1/c$ (recall $1\geq \alpha>0$). 

\section{Proof of Theorem \ref{thm:fklz}} \label{sec:proof2}

We divide the proof into three steps. In Step 1, we transform the problem of general Gaussian approximation in Kolmogorov distance into a standard Gaussian approximation on convex polytopes. In Step 2, we prove a result under a boundedness condition. In Step 3, we perform a truncation and obtain the result for the unbounded case.

\subsection{Step 1: Standard Gaussian approximation on convex polytopes}\label{sec:4.1}

In the first step, we transform the problem from a general Gaussian approximation in Kolmogorov distance to a standard Gaussian approximation on convex polytopes as follows. 
Recall that the diagonal entries of $\Sigma$ are all equal to 1.
Without loss of generality, we assume $\Sigma$ is full rank. If not, the result follows by perturbing the off-diagonal entries of $\Sigma$ by an arbitrarily small amounts and using an approximation argument. 
Write 
\be{
\Sigma=V^{-1} (V^{-1})^\top, \quad (V^{-1})^\top=(v_1,\dots, v_d), \quad |v_j|=1\ \text{for all}\ 1\leq j\leq d.
}
Without loss of generality, we assume $1\geq \alpha, \beta>0$. From \eq{eq:three} and the equality of absolute determinant and volume of paralellepiped, we have that the minimum angle between $v_j$ and $v_k$ for $1\leq j\ne k\leq d$ is $\geq \alpha>0$ and the minimum angle between $v_l$ and the two-dimensional subspace spanned by $\{v_j, v_k\}$ for any triple of distinct indices $1\leq j,k,l\leq d$ is $\geq \beta>0$. We also assume without loss of generality that $\alpha\geq \beta>0$.
Let
\be{
\tilde X_i=VX_i,\ \tilde W=\frac{1}{\sqrt{n}} \sum_{i=1}^n \tilde X_i.
}
We then have 
\be{
\Cov(\tilde W)=V \Cov(W) V^\top,\quad \tilde X_i\cdot v_j=X_{ij}, 
}
\ben{\label{eq:subGauss}
\E \exp(|\tilde X_i\cdot v_j|/B)\leq 2
}
for all $1\leq i\leq n$, $1\leq j\leq d$. 
For fixed unit vectors $v_1,\dots, v_d\in \IR^d$ as above,
let $\mcl K_d(v):=\mcl K_d(v_1,\dots, v_d)$ be the collection 
of convex polytopes of form
\be{
\{x\in \IR^d:\  x\cdot v_j\leq b_j\ \text{for}\ j=1,\dots, d\},
}
where $b_j$, $1\leq j\leq d$, are real numbers.
Define
\be{
\rho:=\sup_{A\in \mcl{K}_d(v)} |\P(\tilde W\in A)- \P(Z\in A)|,\qquad Z\sim N(0,I_d).
}
If $a_j=-\infty$ for all $j$,
then, to prove \eq{eq:fklz}, it suffices to prove
\ben{\label{eq:fklz-1}
\rho\leq \frac{C}{\alpha^2} \norm{\Cov(W)-\Sigma}_{\infty} (\log d)\log n+ \frac{CB^3}{\alpha^2\beta^2\sqrt{n}} (\log d)^2 \sqrt{\log(dn)}(\log n)^4,
}
where $C$ is an absolute constant.
If $a_j$'s are not $-\infty$, the result follows by a similar proof\footnote{Refer to the footnote at the end of the proof of \cref{prop:gauss}.} considering convex polytopes of form
\be{
\{x\in \IR^d:\  x\cdot v_j\leq b_j,\ x\cdot (-v_j)\leq -a_j \ \text{for}\ j=1,\dots, d\}.
}

Without loss of generality, in proving \eq{eq:fklz-1}, we may assume $X_1,\dots, X_n$ and $Z\sim N(0, I_d)$ are independent.

\subsection{Step 2: Bounded case}


In the second step, we assume the boundedness condition that
\be{
\frac{|\tilde X_i\cdot v_j|}{\sqrt{n}}\leq  \delta
}
for all $1\leq i\leq n$, $1\leq j\leq d$ and
prove
\ben{\label{eq:fklz-bdd}
\rho\leq C\left\{\frac{\Delta_0}{\alpha^2}(|\log\Delta_2|\vee1)
+\frac{\Delta_1}{\alpha^2\beta^2} \left[ (\log d)(|\log\Delta_2|\vee1)+\sqrt{\log d}(|\log\Delta_2|\vee1)^{3/2}  \right]+\delta (\log d)^{3/2} \right\},
}
where 
\ba{
\Delta_0&:=(\log d)\|\Sigma-\Cov(W)\|_\infty,&
\Delta_1&:=(\log d)^{3/2} n\delta^3,&
\Delta_2&:=(\frac{\Delta_1}{\alpha^2\beta^2})^2+\delta^2\log d.
}



Recall that we have assumed without loss of generality that $1\geq \alpha\geq\beta>0$. 
From the assumption that $\Sigma_{jj}=1$ and assuming $\norm{\Sigma-\Cov(W)}_\infty\leq 1/2$ (otherwise, \eq{eq:fklz-bdd} trivially holds since $\rho\leq 1$), we have $n\delta^2\geq 1/2$.
Since $\rho\leq 1$, we may assume, without loss of generality,
\ben{\label{wlog}
\delta\sqrt{\log d}\leq \frac{1}{\sqrt{2}}.
}


Let 
\be{
\xi_i=\frac{X_i}{\sqrt{n}},\quad \tilde \xi_i=V\xi_i=\frac{\tilde X_i}{\sqrt{n}}.
}
By the boundedness assumption, we have
\be{
|\xi_{ij}|= |\tilde \xi_i\cdot v_j|\leq \delta.
}

Recall \eq{3}.
By taking the value of $t$ appropriately, we will deduce a recursive inequality for $\rho$ (see \eqref{somewhere}). 
We focus on the situation where $t$ is of the form
\ben{\label{t}
t=t_0+\delta^2\log d,
}
where $t_0$ is a non-negative number whose value will be specified later on (see \eqref{t0}).
We fix $A\in\mcl{K}_d(v)$ (will take $\sup$ in \eqref{somewhere}), write $h:=1_A$ and proceed to bound $|\E T_t\tilde h(\tilde W)|$. 

To bound $|\E T_t\tilde h(\tilde W)|$ in \eq{n03}, we employ the leave-one-out approach in Stein's method for multivariate normal approximation. 
Let $\{\tilde \xi_1',\dots, \tilde \xi_n'\}$ be an independent copy of $\tilde \xi=\{\tilde \xi_1,\dots, \tilde \xi_n\}$ and let $\tilde W^{(i)}=\tilde W-\tilde \xi_i$.
From the independence and mean-zero assumption and using Taylor's expansion, we obtain
\besn{\label{eq:leave}
&\E[\tilde W \cdot \nabla \psi_t(\tilde W)]=\sum_{i=1}^n \E \left[\tilde \xi_i \cdot (\nabla \psi_t(\tilde W) -\nabla \psi_t(\tilde W^{(i)}))  \right]\\
=& \sum_{i=1}^n \E \langle \tilde \xi_i^{\otimes 2}, \nabla^2 \psi_t(\tilde W^{(i)})  \rangle 
+\sum_{i=1}^n \E\langle U\tilde \xi_i^{\otimes 3}, \nabla^3 \psi_t(\tilde W^{(i)}+UU' \tilde \xi_i)\rangle,
}
where $U, U'$ are independent uniform random variables on $[0,1]$ and independent of everything else.
From \eq{n03}, \eq{eq:leave} and the independence between $\tilde \xi_i$ and $\tilde W^{(i)}$, we have
\bes{
\E T_t\tilde h(\tilde W)=& \E\langle I_d, \nabla^2 \psi_t(\tilde W) \rangle - \sum_{i=1}^n \E \langle \tilde \xi_i'^{\otimes 2}, \nabla^2 \psi_t(\tilde W)\rangle\\
&+\sum_{i=1}^n \E \langle \tilde \xi_i'^{\otimes 2}, \nabla^2 \psi_t(\tilde W)\rangle- \sum_{i=1}^n \E \langle \tilde \xi_i'^{\otimes 2}, \nabla^2 \psi_t(\tilde W^{(i)})\rangle\\
&-\sum_{i=1}^n \E\langle U\tilde \xi_i^{\otimes 3}, \nabla^3 \psi_t(\tilde W^{(i)}+UU' \tilde \xi_i)\rangle=R_0+R_1-R_2,
}
where
\be{
R_0=\E \left[ \langle I_d-\sum_{i=1}^n \E \tilde \xi_i^{\otimes 2} , \nabla^2 \psi_t(\tilde W) \rangle   \right],
}
\be{
R_1=\sum_{i=1}^n \E\langle \tilde \xi_i'^{\otimes 2}\otimes \tilde \xi_i, \nabla^3 \psi_t(\tilde W^{(i)}+U \tilde \xi_i)\rangle,
}
\be{
R_2=\sum_{i=1}^n \E\langle U\tilde \xi_i^{\otimes 3}, \nabla^3 \psi_t(\tilde W^{(i)}+UU' \tilde \xi_i)\rangle.
}
We first bound $R_0$. 
Following similar arguments leading to \eq{eq:kernelproof1} and \eq{eq:gauss} (note that $I_d-\Cov(\tilde W)=V(\Sigma-\Cov(W))V^\top$), we obtain
\ban{\label{r0}
|R_0|\leq \frac{C\Delta_0 (|\log t|\vee1)}{\alpha^2}.
}

Next we bound $R_2$.  
Using \eqref{psi-deriv}, we rewrite $R_2$ as
\ba{
R_2
&=-\sum_{i=1}^n\sum_{j,k,l=1}^d\int_t^\infty e^{-3s}\E[\tilde \xi_{ij} \tilde \xi_{ik} \tilde \xi_{il} U \partial_{jkl}h_s(\tilde W^{(i)}(s)+e^{-s}UU'\tilde \xi_i)]ds,
}
where we set $F(s):=e^{-s}F$ for any random vector $F$ in $\mathbb{R}^d$. 

Note that, for any $s\geq t$, we have 
\[
\frac{e^{-s}}{\sqrt{1-e^{-2s}}}\delta
\leq\frac{\delta}{\sqrt{2t}}
\leq\frac{1}{\sqrt{2\log d}},
\]
where the first inequality follows from 
\ben{\label{bh-3.14}
\frac{e^{-s}}{\sqrt{1-e^{-2s}}}\leq\frac{1}{\sqrt{2s}}\qquad\text{for all }s>0,
}
while the second one follows from \eqref{t}.

Let
\[
\kappa:=\sqrt{12\log d-2\log (1-e^{-t})}+\frac{1}{\sqrt{2\log d}},
\]
so that
\ben{\label{kappa}
e^{-(\kappa-e^{-s}\delta/\sqrt{1-e^{-2s}})^2/4}\leq \frac{\sqrt{1-e^{-t}}}{d^3}\leq\frac{\sqrt{t}}{d^3}.
}
Because of $|\tilde \xi_i\cdot v_j|\leq \delta$, $\forall 1\leq j\leq d$, and the minimum angle assumption,
we have 
\ben{\label{eq:eta}
|\tilde \xi_i\cdot v_{jk}|\leq \frac{C\delta}{\alpha},\quad |\tilde \xi_i\cdot v_{jkl}|\leq \frac{C\delta}{\beta}
}
for the relevant unit vectors $v$ appearing in \eq{eq:AHTdeg3} and \eq{eq:vanish3}.
By integration by parts and recalling $h_s(x)=\E 1\{x+\sqrt{1-e^{-2s}}Z\in A\}$ (from \eq{Tsh}) and \eq{at}, we have
\bes{
&\sum_{j,k,l=1}^d 
\tilde \xi_{ij} \tilde \xi_{ik} \tilde \xi_{il}  \partial_{jkl}h_s(\tilde W(s)-e^{-s}\tilde \xi_i+e^{-s}UU'\tilde  \xi_i)\\
=&-\left(\frac{1}{\sqrt{1-e^{-2s}}}\right)^3 \int_{ \IR^d} 1\left\{\frac{\tilde W(s)}{\sqrt{1-e^{-2s}}} -\frac{e^{-s}(1-UU')\tilde  \xi_i}{\sqrt{1-e^{-2s}}} + z\in \frac{A}{\sqrt{1-e^{-2s}}} \right\} \langle  \tilde  \xi_i^{\otimes 3},\nabla^3 \phi_d(z)\rangle dz\\
= & -\left(\frac{1}{\sqrt{1-e^{-2s}}}\right)^3 \int_{\IR^d} 1\left\{\frac{\tilde W(s)}{\sqrt{1-e^{-2s}}} -\frac{e^{-s}(1-UU')\tilde  \xi_i}{\sqrt{1-e^{-2s}}} + z\in \frac{A}{\sqrt{1-e^{-2s}}} \right\} \langle  \tilde  \xi_i^{\otimes 3},\nabla^3 \phi_d(z)\rangle dz
\\&\qquad \qquad \qquad \qquad \qquad \qquad \times  1\left\{\frac{\tilde W(s)}{\sqrt{1-e^{-2s}}}\in \frac{A}{\sqrt{1-e^{-2s}}}(\kappa)\setminus \frac{A}{\sqrt{1-e^{-2s}}}(-\kappa)  \right\}\\
 & -\left(\frac{1}{\sqrt{1-e^{-2s}}}\right)^3 \int_{\IR^d} 1\left\{\frac{\tilde W(s)}{\sqrt{1-e^{-2s}}} -\frac{e^{-s}(1-UU')\tilde  \xi_i}{\sqrt{1-e^{-2s}}} + z\in \frac{A}{\sqrt{1-e^{-2s}}} \right\} \langle  \tilde  \xi_i^{\otimes 3},\nabla^3 \phi_d(z)\rangle dz
\\&\qquad \qquad \qquad \qquad \qquad \qquad\times  1\left\{\frac{\tilde W(s)}{\sqrt{1-e^{-2s}}}\notin \frac{A}{\sqrt{1-e^{-2s}}}(\kappa)\setminus \frac{A}{\sqrt{1-e^{-2s}}}(-\kappa)  \right\}
\\:= & H_1+H_2.
}
For $H_1$, with $\kappa_s:=\sqrt{1-e^{-2s}}\kappa$, by \cref{eq:AHTdeg3,eq:eta} and recalling $1\geq \alpha\geq \beta>0$, we have
\bes{
|H_1|\leq \frac{C \delta^3}{\alpha\beta} \left(\frac{1}{\sqrt{1-e^{-2s}}}\right)^3\frac{(\log d)^{3/2}}{\alpha\beta} 1_{\{\tilde W(s)\in A(\kappa_s)\setminus A(-\kappa_s)\}}.
}
For $H_2$, by \cref{eq:vanish3,kappa,eq:eta} (recall $1\geq \alpha\geq \beta>0$), we have
\bes{
|H_2|\leq &\left(\frac{1}{\sqrt{1-e^{-2s}}}\right)^3 \left|\int_{\IR^d} 1\left(\frac{\tilde W(s)}{\sqrt{1-e^{-2s}}} -\frac{e^{-s}(1-UU')\tilde  \xi_i}{\sqrt{1-e^{-2s}}} + z \in \frac{A}{\sqrt{1-e^{-2s}}} \right) \langle  \tilde  \xi_i^{\otimes 3},\nabla^3 \phi_d(z)\rangle dz\right|
\\&\times  1\bigg(\frac{\tilde W(s)}{\sqrt{1-e^{-2s}}}-\frac{e^{-s}}{\sqrt{1-e^{-2s}}}(1-UU')\tilde  \xi_i\\&\quad\quad\quad\quad\quad\notin \frac{A}{\sqrt{1-e^{-2s}}}(\kappa-\frac{e^{-s}\delta}{\sqrt{1-e^{-2s}}})\setminus \frac{A}{\sqrt{1-e^{-2s}}}(-\kappa+\frac{e^{-s}\delta}{\sqrt{1-e^{-2s}}})  \bigg)\\
\leq &\sup_{x\notin \frac{A}{\sqrt{1-e^{-2s}}}(\kappa-\frac{e^{-s}\delta}{\sqrt{1-e^{-2s}}})\setminus \frac{A}{\sqrt{1-e^{-2s}}}(-\kappa+\frac{e^{-s}\delta}{\sqrt{1-e^{-2s}}})} \left(\frac{1}{\sqrt{1-e^{-2s}}}\right)^3 \left|\int_{\frac{A}{\sqrt{1-e^{-2s}}}}  \langle  \tilde  \xi_i^{\otimes 3},\nabla^3 \phi_d(z-x)\rangle dz\right|\\
\leq & \frac{C \delta^3}{\alpha\beta} \left(\frac{1}{\sqrt{1-e^{-2s}}}\right)^3\sqrt{t}.
}
Thus, 
\bes{
|H_1+H_2|\leq \frac{C \delta^3}{\alpha\beta} \left(\frac{1}{\sqrt{1-e^{-2s}}}\right)^3\left(\frac{(\log d)^{3/2}}{\alpha\beta} 1_{\{\tilde W(s)\in A(\kappa_s)\setminus A(-\kappa_s)\}}+\sqrt{t}\right).
}
Therefore, we obtain
\ben{\label{r2-1}
|R_2|\leq \frac{Cn\delta^3}{\alpha\beta} \int_t^\infty\left(\frac{e^{-s}}{\sqrt{1-e^{-2s}}}\right)^3\left(\frac{(\log d)^{3/2}}{\alpha\beta}\P(\tilde W(s)\in {A}(\kappa_s)\setminus {A}(-\kappa_s))+\sqrt{t}\right)ds.
}
Now, setting $A_s:=e^s A$, we have 
\[
P(\tilde W(s)\in A(\kappa_s)\setminus {A}(-\kappa_s))=P(\tilde W\in A_s(e^s\kappa_s)\setminus A_s(-e^s\kappa_s)).
\]
Because $A_s\in \mcl K_d(v)$, we obtain
\be{
P(\tilde W(s)\in A(\kappa_s)\setminus A(-\kappa_s))\leq \sup_{A\in \mcl K_d(v)} P(Z\in A(e^s \kappa_s)\setminus A(-e^s\kappa_s))+2\rho.
}
Note that Nazarov's inequality in \cref{cor1} gives
\[
P(Z\in A(\epsilon)\setminus A(-\epsilon))\leq C\epsilon\sqrt{\log d}
\]
for any $\epsilon>0$.
Hence, 
\be{
P(\tilde W(s)\in A(\kappa_s)\setminus A(-\kappa_s))
\leq Ce^s\kappa_s \sqrt{\log d}+2\rho.
}
Inserting this bound into \eqref{r2-1} and using \eqref{bh-3.14}, we deduce
\ban{
|R_2|&\leq \frac{Cn\delta^3}{\alpha\beta} \int_t^\infty\left(\frac{e^{-s}}{\sqrt{1-e^{-2s}}}\right)^3\left(\frac{(\log d)^{3/2}}{\alpha\beta}\left(e^s\kappa_s \sqrt{\log d}+\rho\right)+\sqrt{t}\right)ds\nonumber\\
&\leq \frac{Cn\delta^3}{\alpha\beta}\left(\frac{(\log d)^{3/2}}{\alpha\beta}\left((|\log t| \vee 1)\kappa\sqrt{\log d}+\frac{\rho}{\sqrt{t}}\right)+1\right)\nonumber\\
&\leq \frac{C\Delta_1}{\alpha^2\beta^2 }\left((|\log t| \vee 1)\kappa\sqrt{\log d}+\frac{\rho}{\sqrt{t}}\right),\label{r2}
}
where we used the inequality $\kappa\geq1$ to deduce the last line. The same bound holds for $|R_1|$.

From \eqref{3}, \eqref{r0} and \eqref{r2}, we obtain
\ben{\label{somewhere}
\rho
\leq c_0\left[
(|\log t|\vee1)\frac{\Delta_0}{\alpha^2}
+\frac{\Delta_1}{\alpha^2\beta^2 }\left((|\log t| \vee 1)\kappa\sqrt{\log d}+\frac{\rho}{\sqrt{t}}\right)
+e^t \sqrt{t} \log d
\right], 
}
where $c_0>0$ is a universal constant. 
Now we specify the value of $t_0$ in \eqref{t} as
\ben{\label{t0}
t_0=\min\{(\frac{2c_0\Delta_1}{\alpha^2\beta^2})^2, 
\frac{1}{2}\}.
}
If $t_0<1/2$, then from $t=(\frac{2c_0\Delta_1}{\alpha^2 \beta^2})^2+\delta^2 \log d\asymp \Delta_2$ and \eqref{somewhere},
\ben{\label{rho-est}
\frac{1}{2}\rho
\leq C\left[
(|\log \Delta_2|\vee1)\left(
\frac{\Delta_0}{\alpha^2}+\frac{\Delta_1}{\alpha^2\beta^2} \kappa\sqrt{\log d}
\right)
+e^t\sqrt{\Delta_2} \log d
\right]. 
}
Since $t\leq1$ (from $t_0<1/2$ and  \eqref{wlog}), we have $e^t\leq e$. We also have $-\log(1-e^{-t})\leq C(|\log t|\vee1)$ and $\kappa\leq C\sqrt{\log d+|\log \Delta_2|\vee1}$. Hence we obtain
\ba{
&(|\log \Delta_2|\vee1)\frac{\Delta_1}{\alpha^2\beta^2} \kappa\sqrt{\log d}+e^t\sqrt{\Delta_2} \log d\\
&\leq C\left[(|\log \Delta_2|\vee1) \left[\frac{\Delta_1}{\alpha^2\beta^2}\log d  + \frac{\Delta_1}{\alpha^2\beta^2} \sqrt{\log d} \sqrt{ |\log \Delta_2|\vee 1} \right]  +\sqrt{\Delta_2} \log d\right]\\
&\leq C\left[(|\log \Delta_2|\vee1) \left[\frac{\Delta_1}{\alpha^2\beta^2}\log d  + \frac{\Delta_1}{\alpha^2\beta^2} \sqrt{\log d} \sqrt{ |\log \Delta_2|\vee 1} \right]  +\delta (\log d)^{3/2}\right].
}
Inserting this bound into \eqref{rho-est}, we obtain
\ben{\label{aim}
\rho\leq C\left\{\frac{\Delta_0}{\alpha^2}(|\log\Delta_2|\vee1)
+\frac{\Delta_1}{\alpha^2\beta^2} \left[ (\log d)(|\log\Delta_2|\vee1)+\sqrt{\log d}(|\log\Delta_2|\vee1)^{3/2}  \right]+\delta (\log d)^{3/2} \right\},
} 
which is the desired bound \eq{eq:fklz-bdd}. 
If $t_0=1/2$, then we have $2\sqrt{2} c_0 \Delta_1/(\alpha^2\beta^2)\geq1\geq\rho$, so \eqref{aim} is satisfied with $C=2\sqrt{2} c_0$. 
Overall, we complete the proof of \eq{eq:fklz-bdd}. 

\subsection{Step 3: Unbounded case}

Note that in proving \eq{eq:fklz}, we may assume $\norm{\Cov(W)-\Sigma}_\infty\leq 1/2$ since $\rho(W,G)\leq 1$. We then have $2B^2\geq 1/2$ (cf. \eq{eq:star3}).
If $1>B\geq 1/2$ and we increase $B$ to 1, the condition of \cref{thm:fklz} still holds and the conclusion \eq{eq:fklz} is invariant.
Therefore, we can assume without loss of generality that $B\geq 1$.
Recall that we have assumed without loss of generality that $1\geq \alpha\geq \beta>0$.
Without loss of generality, we may further assume
\ben{\label{wlog2}
\frac{10^4 B^6(\log n)^6(\log d)^3}{\alpha^4\beta^4 n}\leq 1;
}
otherwise, the bound \eq{eq:fklz} trivially holds true with sufficiently large $C$. 

Set $\kappa_n:=2B\log n$. For $i=1,\dots,n$ and $j=1,\dots,d$, define 
\[
\hat {X}_{ij}:=X_{ij}1_{\{|X_{ij}|\leq\kappa_n\}}-\E X_{ij}1_{\{|X_{ij}|\leq\kappa_n\}},
\] 
and set $\hat {W}:=n^{-1/2}\sum_{i=1}^n\hat {X}_i$ with $\hat{X}_i=(\hat{X}_{i1},\dots,\hat{X}_{id})^\top$. 
Note that $\max_{i,j}|\hat{X}_{ij}|\leq2\kappa_n$. 
Also, observe that
\[
X_{ij}-\hat {X}_{ij}=X_{ij}1_{\{|X_{ij}|>\kappa_n\}}-\E X_{ij}1_{\{|X_{ij}|>\kappa_n\}}.
\]
Therefore, for any integer $p\geq2$,
\besn{\label{eq:lp}
\E|X_{ij}-\hat {X}_{ij}|^p&\leq2^p\E[|X_{ij}|^p1_{\{|X_{ij}|>\kappa_n\}}]
\leq2^{p+1}p!e^{-\kappa_n/B}(\kappa_n+B)^p\\
&\leq\frac{2^{2p+1}p!\kappa_n^p}{n^2}
=\frac{p!(4\kappa_n)^{p-2}(8\kappa_n/n)^2}{2},
}
where the second inequality follows by Lemma 5.4 in \cite{Ko21}. Thus, by the Bernstein inequality (cf.~\citealp[Lemma 2.2.11]{VW96}), we have for any $t>0$
\ba{
P(|W_j-\hat W_j|>t)\leq2\exp\left(-\frac{1}{2}\frac{t^2}{(8\kappa_n/n)^2+(4\kappa_n/\sqrt n)t}\right).
}
Taking $t=2(8\kappa_n/\sqrt n)\log(dn)$, we obtain
\ba{
P(|W_j-\hat W_j|>t)\leq2\exp\left(-\log(dn)\right)\leq\frac{2}{dn}.
}
Hence, by the union bound, $P(\max_j|W_j-\hat W_j|>32B(\log n)\log(dn)/\sqrt n)\leq2/n$. From this estimate and the Gaussian anti-concentration inequality in \cref{l1}, we have
\bes{
P(W\leq b)-P(G\leq b)\leq & P(\hat W\leq b+\frac{32B(\log n)\log(dn)}{\sqrt n}\textbf{1}_d)-P(G\leq b+\frac{32B(\log n)\log(dn)}{\sqrt n}\textbf{1}_d)\\
& +P(\max_j |W_j-\hat W_j|>\frac{32B(\log n)\log(dn)}{\sqrt n})\\
&+P(G\leq b+\frac{32B(\log n)\log(dn)}{\sqrt n}\textbf{1}_d)-P(G\leq b)\\
\leq & C\left(\hat\rho+\frac{2}{n}+\frac{B \sqrt{\log d}\log(dn)}{\sqrt n}\log n\right),
}
where
\[
\hat\rho:=\sup_{b\in \IR^d} |P(\hat W\leq b)- P(G\leq b)|.
\]
From a similar argument for the lower bound, we conclude
\ben{\label{t2-0}
\sup_{b\in \IR^d} |P(W\leq b)- P(G\leq b)|
\leq C\left(\frac{2}{n}+\frac{B \sqrt{\log d}\log(dn)}{\sqrt n}\log n+\hat\rho\right),
}

To estimate $\hat\rho$, we apply \eq{eq:fklz-bdd} to $\hat W$ with $\delta=2\kappa_n/\sqrt{n}$.
For this purpose, we bound the quantities in \eq{eq:fklz-bdd}. 
First, observe that 
\ba{
\norm{\Cov(W)-\Cov(\hat W)}_\infty
&=\max_{j,k}\left|\frac{1}{n}\sum_{i=1}^n(\E[X_{ij}X_{ik}]-\E[\hat X_{ij}\hat X_{ik}])\right|\\
&\leq\max_{i,j,k}(\sqrt{\E|X_{ij}-\hat X_{ij}|^2\E|X_{ik}|^2}+\sqrt{\E|\hat X_{ij}|^2\E|X_{ik}-\hat X_{ik}|^2})\\
&\leq C\frac{\kappa_n}{n}B\leq C\frac{B^2}{\sqrt n},
}
where we used \eq{eq:lp} and $\E X_{ij}^2\leq 2 B^2$ from the sub-exponential condition (cf. \eq{eq:star3}).
Hence
\ban{\label{t2-1a}
\Delta_0
\leq C(\log d)\left[\norm{\Cov(W)-\Sigma}_\infty+\frac{B^2}{\sqrt n} \right].
}

We also have
\ban{\label{t2-1b}
\Delta_1\leq \frac{4^3B^3}{\sqrt{n}} (\log n)^{3}(\log d)^{3/2}.
}
We then have (recall $B\geq 1$)
\ben{\label{t2-2}
\Delta_2= (\frac{\Delta_1}{\alpha^2\beta^2})^2+\frac{4^2B^2(\log n)^2\log d}{n}
\leq \frac{10^4 B^6(\log n)^6(\log d)^3}{\alpha^4\beta^4 n}.
}
In particular, we have $\Delta_2\leq1$ by \eqref{wlog2}, so we obtain
\ben{\label{t2-3}
|\log\Delta_2|=-\log\Delta_2\leq-\log\frac{4^2B^2(\log n)^2\log d}{n}\leq\log n.
}
Inserting \eqref{t2-1a}--\eqref{t2-3} into \eqref{eq:fklz-bdd} and combining with \eq{t2-0}, we complete the proof of \eq{eq:fklz}.


\section{Proof of lemmas}\label{sec:prooflem} 

\begin{proof}[Proof of \cref{lem:AHT}]

We first introduce some notation to denote the facets of the convex polytope $A$, the facets of facets, and so on.
Each unit vector $v_j$ in $V_1$ corresponds to a possible facet of $A$, denoted by $F_j$. When a convex polytope is $D$-dimensional, we only consider its $(D-1)$-dimensional facets ($D-2$ or even lower-dimensional corners do not enter into consideration after applying the divergence theorem below).
Then, $F_j$ is a $(d-1)$-dimensional convex polytope with at most $d-1$ facets, denoted by $F_{jk}$, created by intersecting with other $F_k, k\ne j$. Note that $F_{jk}=F_{kj}$. Finally, $F_{jk}$ is a $(d-2)$-dimensional convex polytope with at most $d-2$ facets, denoted by $F_{jkl}$, created by intersecting with $F_l, l\ne j, k$. Note that $F_{jkl}$ is invariant under the permutation of its sub-indices.



For the sake of discussion below, we assume $A$ is sufficiently regular, that is, no more than two facets intersect on a $(d-2)$-dimensional subspace, etc. 
This will be guaranteed by \cref{rankassumption1,rankassumption2,rankassumption3} (see \cref{lem:disjoint}).
More precisely, without loss of generality, we make the following assumptions. 
For  any $1\leq j<k\leq d$, if $\rank \begin{bmatrix} v_j & v_k \end{bmatrix}<2$, then
\begin{equation}\label{rankassumption1}
\rank \begin{bmatrix} v_j & v_k\end{bmatrix}<\rank \begin{bmatrix} v_j & v_k\\ b_j&b_k \end{bmatrix}. 
\end{equation}
For  any $1\leq j<k<l\leq d$, if $\rank \begin{bmatrix} v_j & v_k & v_l\end{bmatrix}<3$, then
\begin{equation}\label{rankassumption2}
\rank \begin{bmatrix} v_j & v_k & v_l\end{bmatrix}<\rank \begin{bmatrix} v_j & v_k & v_l\\ b_j&b_k&b_l \end{bmatrix}. 
\end{equation}
For  any $1\leq j<k<l<s\leq d$, if $\rank \begin{bmatrix} v_j & v_k & v_l& v_s\end{bmatrix}<4$, then
\begin{equation}\label{rankassumption3}
\rank \begin{bmatrix} v_j & v_k & v_l& v_s\end{bmatrix}<\rank \begin{bmatrix} v_j & v_k & v_l& v_s\\ b_j&b_k&b_l&b_s \end{bmatrix}. 
\end{equation}
In fact, if 
$$\rank \begin{bmatrix} v_j & v_k \end{bmatrix}=\rank \begin{bmatrix} v_j & v_k \\ b_j&b_k\end{bmatrix}<2,$$
or
$$\rank \begin{bmatrix} v_j & v_k & v_l\end{bmatrix}=\rank \begin{bmatrix} v_j & v_k & v_l\\ b_j&b_k&b_l\end{bmatrix}<3,$$
or
$$\rank \begin{bmatrix} v_j & v_k & v_l& v_s\end{bmatrix}=\rank \begin{bmatrix} v_j & v_k & v_l& v_s\\ b_j&b_k&b_l&b_s\end{bmatrix}<4,$$ we can make an infinitesimally small perturbation of $b_j$, $j=1,...,d$ to make \cref{rankassumption1,rankassumption2,rankassumption3} hold and use an approximation argument to draw the same conclusion (note that the upper bounds in \cref{lem:AHT} are independent of the $b$'s).


\paragraph{Proof of \eq{eq:AHTdeg1}.}
We first consider
\be{
\int_A \langle u, \nabla \phi_d (z) \rangle dz.
}
From the divergence theorem\footnote{See, for example, Theorem 5.16 in \cite{EvGa15}.}, we obtain
\be{
\int_A \langle u, \nabla \phi_d (z) \rangle dz=
\sum_{F_j:\ \text{facet of}\ A} (u\cdot v_j)
\int_{F_j} \phi_d (z)  \mcl H^{d-1}(dz),
}
where $\mcl H^{d-1}$ is the $(d-1)$-dimensional Hausdorff measure.
From Nazarov's inequality (cf. \cref{cor1} or \cite{chernozhukov2017detailed}), we obtain
the desired bound.

\paragraph{Proof of \eq{eq:AHTdeg2}.} Next, we consider 
\be{
\int_A \langle M, \nabla^2 \phi_d (z)\rangle dz,
}
where $M$ is a $d\times d$ matrix.
Given any unit vector $v_j$, we take vectors $w_{j}^{(1)},\dots,w^{(d-1)}_{j}\in\mathbb R^d$ such that $w_j^{(1)},\dots,w_j^{(d-1)},v_j$ constitute an orthonormal basis of $\mathbb R^d$. Set $U_j:=(w_j^{(1)},\dots,w_j^{(d-1)})\in\mathbb R^{d\times(d-1)}$ and $\bar U_j:=(U_j,v_j)\in\mathbb R^{d\times d}$. By construction, $\bar U_j$ is an orthogonal matrix. Hence, for any $z\in\mathbb R^d$, 
\begin{equation}
   \label{eq:2dri3}
  \begin{aligned}
|z|^2=|\bar U_j^\top z|^2=|U_j^\top z|^2+|v_j\cdot z|^2. 
  \end{aligned}
\end{equation}
Using the divergence theorem in the second equation and the orthogonal invariance of inner products and $\phi_d(\cdot)$ in the third equation, we obtain
\begin{equation*}
  \begin{split}
    &\int_A \langle M, \nabla^2 \phi_d (z)\rangle dz= \int_A \nabla \cdot (M \nabla \phi_d(z)))dz\\
    &=-\sum_{F_j:\ \text{facet of}\ A}  \int_{F_j}   \langle v_j , M z\rangle \phi_d(z) \mcl H^{d-1}(dz)\\
    &=-\sum_{F_j:\ \text{facet of}\ A}  \int_{F_j} \langle \bar{U}_j^\top M^\top v_j, \bar{U}_j^\top z\rangle \phi_d(\bar{U}_j^\top z) \mcl H^{d-1}(dz)\\
&=-\sum_{F_j:\ \text{facet of}\ A}  \int_{F_j} ((v_j^\top M^\top v_j)(v_j \cdot z)+ \langle {U}_j^\top M^\top v_j, {U}_j^\top z\rangle  ) \phi_{d-1}({U}_j^\top z)\phi_1(v_j\cdot z) \mcl H^{d-1}(dz)
\\&=I+II,
  \end{split}
\end{equation*}
where 
$$I=-\sum_{F_j:\ \text{facet of}\ A} ( v_j^\top M v_j)\int_{F_j} (v_{j}\cdot z) \phi_{d}(\bar{U}_{j}^\top z) \mcl H^{d-1}(dz);$$
$$II=-\sum_{F_j:\ \text{facet of}\ A}  \int_{F_j} \langle {U}_j^\top M^\top v_j, {U}_j^\top z\rangle   \phi_{d-1}({U}_j^\top z)\phi_1(v_j\cdot z) \mcl H^{d-1}(dz).$$
From the first equality in \cref{eq:2dri3}, we obtain 
\begin{equation}
   \label{eq:2dri4}
  \begin{aligned}
    I&=-\sum_{F_j:\ \text{facet of}\ A} ( v_j^\top M v_j) \int_{F_j} (v_{j}\cdot z) \phi_{d}( z) \mcl H^{d-1}(dz).
  \end{aligned}
\end{equation}

For $II$, Let $x_{j}$ be a point in $F_j$. Then, by the definition of $v_j$ and $F_j$, we have 
$$b_j=v_j\cdot x_j=v_j\cdot z,\quad   \forall z\in F_{j}.$$
From a change of variables\footnote{For example, applying Theorem 3.9 in \cite{EvGa15} with $n=d-1,m=d,f(y)=U_jy$ and $g(y)=\phi_{d-1}(y)1_{U_j^\top F_j}(y)$.}, we obtain 
\begin{equation}
   \label{eq:twicedri}
  \begin{aligned}
    II&=-\sum_{F_j:\ \text{facet of}\ A}  \phi_1(v_j\cdot x_{j})\int_{F_j} \langle {U}_j^\top M^\top v_j, {U}_j^\top z\rangle  \phi_{d-1}({U}_j^\top z) \mcl H^{d-1}(dz) 
\\&=-\sum_{F_j:\ \text{facet of}\ A}  \phi_1(v_j\cdot x_{j})\int_{U_j^\top F_j} \langle {U}_j^\top M^\top v_j,  z\rangle  \phi_{d-1}(z) dz
\\&=\sum_{F_j:\ \text{facet of}\ A}  \phi_1(v_j\cdot x_{j})\int_{U_j^\top F_j} \langle {U}_j^\top M^\top v_j,  \nabla\phi_{d-1}(z)\rangle   dz.
  \end{aligned}
\end{equation}
Recall that $F_j$ is a $(d-1)$-dimensional polytope with at most $d-1$ facets $F_{jk}$ with out-pointing unit normal $v_{jk}\in \IR^d$ as defined above \cref{lem:AHT}. Given any $v_j, v_{jk}$,
we take vectors $w_{jk}^{(1)},\dots,w^{(d-2)}_{jk}\in\mathbb R^d$ such that $w^{(1)}_{jk},\dots,w^{(d-2)}_{jk},v_j,v_{jk}$ constitute an orthonormal basis of $\mathbb R^d$. Set $$U_{jk}:=(w_{jk}^{(1)},\dots,w_{jk}^{(d-2)})\in\mathbb R^{d\times(d-2)};$$
$$U_{jk}^*:=(U_j^\top w_{jk}^{(1)},\dots,U_j^\top w_{jk}^{(d-2)})\in\mathbb R^{(d-1)\times(d-2)}.$$
By the divergence theorem and change of variables again, and using $(U_{jk}^*)^\top U_j^\top=U_{jk}^\top$ (both sides equal to the same projection from $\IR^d$ to $\IR^{d-2}$), we have
\begin{equation}
   \label{eq:2dri1}
  \begin{aligned}
    II \quad &=\sum_{F_j:\ \text{facet of}\ A} \sum_{F_{jk}:\ \text{facet of}\ F_{j}}(v_{jk}^\top M^\top v_j)\phi_1(v_j\cdot x_{j})\phi_1(v_{jk}\cdot x_{jk})\int_{(U_{jk}^*)^\top U_j^\top F_j}    \phi_{d-2}(z) dz\\&=\sum_{F_j:\ \text{facet of}\ A} \sum_{F_{jk}:\ \text{facet of}\ F_{j}}(v_{jk}^\top M^\top v_j)\phi_1(v_j\cdot x_{j})\phi_1(v_{jk}\cdot x_{jk})\int_{U_{jk}^\top F_j}    \phi_{d-2}(z) dz,
  \end{aligned}
\end{equation}
where $x_{jk}$ is a point in $F_{jk}$ and we used $(U_j^\top M^\top v_j)\cdot (U_j^\top v_{jk})=(\bar U_j^\top M^\top v_j)\cdot (\bar U_j^\top v_{jk})=v_{jk}^\top M^\top v_j$. Similarly, by change of variables, we also have 
\be{
\phi_1(v_j\cdot x_{j})\phi_1(v_{jk}\cdot x_{jk})\int_{U_{jk}^\top F_j}    \phi_{d-2}(z) dz=\int_{F_{jk}} \phi_d(z) \mcl H^{d-2}(dz).
}
Therefore,
\begin{equation}
   \label{eq:2dri2}
  \begin{aligned}
  II=\sum_{F_j:\ \text{facet of}\ A}\sum_{F_{jk}:\ \text{facet of}\ F_j} (v_{j}^\top M v_{jk}) \int_{F_{jk}} \phi_d(z) \mcl H^{d-2}(dz).
  \end{aligned}
\end{equation}
Combining \cref{eq:2dri4,eq:2dri2}, we obtain
\begin{equation}
   \label{eq:2dri5}
  \begin{aligned}
    \int_A \langle M, \nabla^2 \phi_d (z)\rangle dz= \sum_{F_j:\ \text{facet of}\ A}(R_{1j} + R_{2j}),
  \end{aligned}
\end{equation}
where
\ben{\label{eq:r1j}
R_{1j}=( v_j^\top M v_j)\int_{F_j} \partial_{v_j} \phi_d(z) \mathcal{H}^{d-1} (dz),
}
$\partial_{v_j}\phi_d(z)=v_j\cdot \nabla \phi_d(z)$ denotes the directional derivative,
and
$$R_{2j}=\sum_{F_{jk}:\ \text{facet of}\ F_j} (v_{j}^\top M v_{jk}) \int_{F_{jk}} \phi_d(z) \mcl H^{d-2}(dz).$$




Following the proof of Eq.(1) of \cite{chernozhukov2017detailed} (bounding a standard normal density by its tail probability $\phi_1(r)\leq (r+1)\int_r^\infty \phi_1(x)dx$ for $r\geq 0$), we have
\ben{\label{eq:normaltail}
\left|\int_{F_j} \partial_{v_j} \phi_d(z) \mcl H^{d-1}(dz)\right|\leq\text{dist}(0,F_j)(\text{dist}(0,F_j)+1)\gamma_d(S_j),
}
where $\text{dist}(0, F_j)$ is the distance from 0 to the hyperplane containing $F_j$, $\gamma_d$ is the $d$-dimensional standard Gaussian measure, $S_j:=\{y+t v_j: y\in \rel(F_j), t> 0\}$ and $\rel(F_j)$ is the relative interior of $F_j$.
Following the arguments on p.4 of \cite{chernozhukov2017detailed}, we have (using the normal tail bound and \eq{eq:r1j} for the first sum and \eq{eq:normaltail} and disjointness of $S_j$'s proved in \cref{lem:Sjdis} for the second sum)
\besn{\label{eq:R1jtruncation}
  \left|\sum_{F_j:\  \text{facet of}\  A}  R_{1j}\right|=&\left|\sum_{\text{dist}(0,F_j)>\sqrt{4\log d}} R_{1j}+\sum_{\text{dist}(0,F_j)\leq \sqrt{4\log d}} R_{1j}\right|\\
  \leq& C d\frac{1}{d} \max_{v\in V_1} |v^\top M v| + C(\sqrt{\log d})^2 \max_{v\in V_1} |v^\top M v|   \\
  \leq& C \max_{v\in V_1} |v^\top M v|\log d.
}

From $F_{jk}=F_{kj}$, we write
\ben{\label{eq:R2j}
\sum_{F_j:\  \text{facet of}\  A}  R_{2j}=\sum_{F_{jk}:\  \text{facet of}\  F_j\atop 1\leq j<k\leq d} \left[  v_j^\top M  v_{jk}+v_k^\top M v_{kj} \right] \int_{F_{jk}} \phi_d(z) \mcl H^{d-2}(dz).
}
Let
\ben{\label{eq:Sjk}
S_{jk}:=\{y+s v_j+t v_k: y\in \rel(F_{jk}), s> 0, t> 0\}   ,\ \text{for}\  1\leq j< k\leq d\ \text{and $F_{jk}$ is a facet of $F_j$}.
}
We will prove in \cref{sjkdisjoint} that $S_{jk}$'s are disjoint. 



Now we consider each summand on the right-hand side of \eq{eq:R2j}.
Denote by $\bar x_{jk}$ the projection of 0 to the $(d-2)$-dimensional affine subspace containing $F_{jk}$.
Note that $\bar x_{jk}\in span\{v_j, v_k\}$, where $span$ denotes the subspace generated by the vectors.
We define a new set $\bar S_{jk}$ by (possibly) decreasing and rotating $S_{jk}$ around $\bar x_{jk}$ in the two-dimensional subspace spanned by $v_j$ and $v_k$ as follows. 
Consider the circle centered at $\bar x_{jk}$ with unit radius in $span\{v_j, v_k\}$. Extend the line from 0 to $\bar x_{jk}$ until it intersects with the circle, call the intersection the pole. If $0=\bar x_{jk}$, then define $v_j$ as the pole. The two points $\bar x_{jk}+v_j$ and $\bar x_{jk}+v_k$ form an arc with degree less than $\pi$ on the circle. We assume without loss of generality that the degree of the arc is less than or equal to $\pi/2$. Otherwise, we rotate $v_k$ and decrease $S_{jk}$. This does not affect the conclusion \eq{Fjkalter} below.
If the arc contains the pole, then do nothing. If the arc does not contain the pole, pick between $\bar x_{jk}+v_j$ and $\bar x_{jk}+v_k$ the point closer to the pole. If the distances are equal, pick either point. Say this point is $\bar x_{jk}+v_j$. Then, we rotate the arc so that $\bar x_{jk}+v_j$ equals the pole. Denote by $\bar v_j$ and $\bar v_k$ the rotated versions of $v_j$ and $v_k$ respectively after the above rotation. Define
\be{
\bar S_{jk}:=\{y+s \bar v_j+t \bar v_k: y\in \rel(F_{jk}), s>0, t>0\}.
}
Recall that the arc degree is less than or equal to $\pi/2$. The above rotation moves every point on the arc closer to the pole, hence increases their Euclidean norms.
Therefore, we have
\ben{\label{eq:gammabars}
\gamma_d(\bar S_{jk})\leq \gamma_d(S_{jk}).
}

For each integral on the right-hand side of \eq{eq:R2j},
the case $\text{dist}(0, F_{jk})>2\sqrt{\log d}$ can be dealt with easily using the normal tail bound as in \eq{eq:R1jtruncation}. 
Let $u_{jk}=\bar x_{jk}/|\bar x_{jk}|$. If $\bar x_{jk}=0$, let $u_{jk}=v_j$. Let $u_{jk}^\perp$ be a unit vector in $span\{v_j, v_k\}$ and orthogonal to $u_{jk}$.
For the case $\text{dist}(0, F_{jk})\leq 2\sqrt{\log d}$, we bound two standard normal densities by normal tail probabilities in two orthogonal directions to $F_{jk}$ and obtain
\ben{\label{eq:star1}
\left|\int_{F_{jk}} \phi_d(z) \mcl H^{d-2}(dz)\right| \leq (\text{dist}(0, F_{jk})+1)^2 \gamma_d(N_{F_{jk}}),
}
where 
\ben{\label{eq:NFjk}
N_{F_{jk}}:=\{y+s(u_{jk}+u_{jk}^\perp) + t(u_{jk}-u_{jk}^\perp): y\in \rel(F_{jk}), s>0, t>0\}.
}
Suppose the angle between $v_j$ and $v_k$ is $\theta>0$. Then, computing standard Gaussian probabilities by first integrating along the coordinate $u_{jk}$, we have, with $t=|\bar x_{jk}|$,
\be{
\frac{\gamma_d(N_{F_{jk}})}{\gamma_d(\bar S_{jk})}=\frac{\int_t^\infty \phi_1(u) \gamma_1[-(u-t), (u-t)] du}{\int_t^\infty \phi_1(u) \gamma_1[a(u,t) ,b(u,t)] du},
}
where $[a(u,t) ,b(u,t)]$ is an interval containing $0$ and with length larger than $(u-t)\theta$. Because Gaussian density is decreasing in the absolute value of its argument, we have $\gamma_1[-(u-t), (u-t)]/\gamma_1[a(u,t) ,b(u,t)] \leq C/\theta$, hence
\ben{\label{Fjkalter}
\gamma_d(N_{F_{jk}})\leq \frac{C}{\theta} \gamma_d(\bar S_{jk})\leq \frac{C}{\theta}\gamma_d( S_{jk}),
}
where we used \eq{eq:gammabars} in the last inequality.

If $\theta>1/10$, then from \eq{eq:star1} and \eq{Fjkalter}, we have (recall we have assumed $\text{dist}(0, F_{jk})\leq 2\sqrt{\log d}$)
\ben{\label{Fjk}
\left|  v_j^\top M  v_{jk}+v_k^\top M v_{kj} \right| \int_{F_{jk}} \phi_d(z) \mcl H^{d-2}(dz)\leq C(\log d)\Delta \gamma_d(S_{jk}),
}
where 
\be{
\Delta:=\max_{j,k} (|v_j^\top M v_j|+|v_j^\top M v_{jk}|+ |v_{jk}^\top M v_{jk}| )
}
 and the max is taken over all $1\leq j\ne k\leq d$ such that $v_k\ne - v_j$.
If $\theta\leq 1/10$, then
\begin{equation}
   \label{eq:dri6}
  \begin{aligned}
    v_j^\top M  v_{jk}+v_k^\top M v_{kj}& = (v_j-v_k)^\top M  v_{jk}+v_k^\top M (v_{kj}+v_{jk})\\
                                     &= (O(\theta)v_{jk}+O(\theta^2) v_j)^\top M  v_{jk}+v_k^\top M (O(\theta)v_k+ O(\theta^2) v_{kj})\\
                                     &=O(\theta)\Delta,
  \end{aligned}
\end{equation}
and hence \eq{Fjk} still holds.
Then, from the facts that $S_{jk}$'s are disjoint and $\gamma_d$ is a probability measure, the contribution from those $1\leq j<k\leq d$ such that $\text{dist}(0, F_{jk})\leq 2\sqrt{\log d}$ to the right-hand side of \eq{eq:R2j} is bounded by $$C\Delta \log d.$$
Combining the two cases, we obtain
\bes{
&\left|\sum_{F_j:\  \text{facet of}\  A} R_{2j}\right|\leq C \Delta\log d.
}
This, together with \eq{eq:R1jtruncation}, gives the desired bound  \cref{eq:AHTdeg2} in \cref{lem:AHT}.

\paragraph{Proof of \eq{eq:AHTdeg3}.}
Finally, we consider
\be{
\int_A\langle T, \nabla^3 \phi_d (z) \rangle dz.
}
Without loss of generality, we assume $1\geq \alpha\geq \beta>0$.
Following a similar and more tedious proof (that we omit) leading to \eq{eq:2dri5}, we obtain
\besn{\label{eq:thirdderi}
&\int_A\langle T, \nabla^3 \phi_d (z) \rangle dz\\
=& \sum_{F_j:\  \text{facet of}\  A} \langle T, v_j\otimes v_j \otimes v_j\rangle \int_{F_j} \partial^2_{v_j}\phi_d(z)\mcl H^{d-1}(dz)\\
&+\sum_{F_j:\  \text{facet of}\  A}  \sum_{F_{jk}:\  \text{facet of}\  F_j} \langle T, v_j\otimes v_j \otimes v_{jk}\rangle \int_{F_{jk}} \partial_{v_{j}}\phi_d(z)\mcl H^{d-2}(dz)\\
&+\sum_{F_j:\  \text{facet of}\  A} \sum_{F_{jk}:\  \text{facet of}\  F_j}  
 \langle T, v_j\otimes v_{jk} \otimes v_j\rangle\int_{F_{jk}} \partial_{v_{j}}\phi_d(z)\mcl H^{d-2}(dz)\\
&+\sum_{F_j:\  \text{facet of}\  A} \sum_{F_{jk}:\  \text{facet of}\  F_j}   \langle T, v_j\otimes v_{jk} \otimes v_{jk}\rangle\int_{F_{jk}} \partial_{v_{jk}}\phi_d(z)\mcl H^{d-2}(dz)\\
&+\sum_{F_j:\  \text{facet of}\  A} \sum_{F_{jk}:\  \text{facet of}\  F_j}  \sum_{F_{jkl}:\  \text{facet of}\  F_{jk}} \langle T, v_j\otimes v_{jk} \otimes v_{jkl}\rangle\int_{F_{jkl}} \phi_d(z)\mcl H^{d-3}(dz).
}
Following similar arguments as for the second-derivative bound, the first three terms of \eq{eq:thirdderi} are bounded by 
\besn{
\label{123term}
 {C (\log d)^{3/2}} \max_{v_1\in V_1, v_2,v_3\in V_1\cup V_2} |\langle T, v_1 \otimes v_2 \otimes v_3\rangle|,
}
and the fourth term is bounded by 
\ben{
\label{4thterm}
 \frac{C (\log d)^{3/2}}{\alpha} \max_{v_1\in V_1,v_2,v_3\in V_2} |\langle T, v_1 \otimes v_2 \otimes v_3\rangle|,
}
where $\alpha$ is the minimum angle between two distinct $v_j, v_k$. We note that the trick in \cref{eq:dri6,Fjk} does not work for the fourth term, hence the denominator $\alpha$ in its upper bound \cref{4thterm}.
Because $\alpha\geq \beta$,  \cref{123term,4thterm} are bounded by the right-hand side of \eq{eq:AHTdeg3}.

Following the arguments leading to \cref{eq:R2j}, the last term in \eq{eq:thirdderi} can be written as
\begin{equation}\label{Rjkl}
\begin{aligned}
&\sum_{F_{jkl}:\ \text{facet of}\ F_{jk}\atop 1\leq j< k<l\leq d} R_{jkl}\\
     :=&\sum_{F_{jkl}:\ \text{facet of}\ F_{jk}\atop 1\leq j< k<l\leq d} \bigg[ \langle T,
   v_j\otimes v_{jk} \otimes v_{jkl}\rangle+\langle T,v_k\otimes v_{kj} \otimes v_{jkl}\rangle+\langle T,v_j\otimes v_{jl} \otimes v_{jlk}\rangle
   \\&+\langle T,v_l\otimes v_{lj} \otimes v_{jlk}\rangle+\langle T,v_k\otimes v_{kl} \otimes v_{lkj}\rangle+\langle T,v_l\otimes v_{lk} \otimes v_{lkj}
   \rangle \bigg] \int_{F_{jkl}} \phi_d(z) \mcl{H}^{d-3}(dz).
\end{aligned}
\end{equation}
Let
\besn{\label{eq:Sjkl}
S_{jkl}:=&\{y+s v_j+t v_k+r v_l: y\in \rel(F_{jkl}), s> 0, t> 0,r>0\} \\
&\qquad\qquad\qquad \text{for}\  1\leq j< k<l\leq d\ \text{and $F_{jkl}$ is a facet of $F_{jk}$}.
}
We will prove in \cref{sjkldisjoint} that $S_{jkl}$'s are disjoint. 

Following similar but more technical arguments leading to \eq{Fjk} (we leave the details to \cref{sec:Fjkl2}), we obtain
\begin{equation}
   \label{eq:Fjkl2}
  \begin{aligned}
\int_{F_{jkl}} \phi_d(z) \mcl{H}^{d-3}(dz) \leq \frac{C(\log d)^{3/2}}{\alpha\beta} \gamma_d(S_{jkl}).
  \end{aligned}
\end{equation}
Using the fact that $S_{jkl}$'s are disjoint and $\gamma_d$ is a probability measure, we obtain
\begin{equation}
   \label{eq:sumR}
  \begin{aligned}
  &\sum_{F_{jkl}:\ \text{facet of}\ F_{jk}\atop 1\leq j< k<l\leq d} R_{jkl}\leq \frac{C (\log d)^{3/2}}{\alpha\beta} \max_{v_1\in V_1, v_2\in V_2,\atop v_3\in V_3} |\langle T, v_1 \otimes v_2 \otimes v_3\rangle|.
  \end{aligned}
\end{equation}
Therefore, we finish the proof of \cref{eq:AHTdeg3} in  \cref{lem:AHT}.
\end{proof}

\begin{proof}[Proof of \cref{lem:vanish}]
From the divergence theorem, we obtain
\ba{
\int_A\langle u,\nabla\phi_d(z-x)\rangle dz=\sum_{F_j:\ \text{facet of}\ A}(u\cdot v_j)\int_{F_j}\phi_d(z-x)\mcl H^{d-1}(dz),
}
where $\mcl H^{d-1}$ is the $(d-1)$-dimensional Hausdorff measure. 
Recall $A$ has at most $d$ facets.
Therefore, we complete the proof once we show that
\be{
\int_{F_j}\phi_d(z-x)\mcl H^{d-1}(dz)\leq \phi_1(\kappa)
}
for any $F_j$ and $x\notin A(\kappa)\setminus A(-\kappa)$. Observe that we have $x\notin A(\kappa)$ or $x\in A(-\kappa)$. We separately consider these two cases.

\noindent\textbf{Case 1:} $x\notin A(\kappa)$.  
In this case, we have $x\cdot v_k>b_k+\kappa$ for some $k\in\{1,\dots,p\}$. 
Take vectors $w_1,\dots,w_{d-1}\in\mathbb R^d$ such that $w_1,\dots,w_{d-1},v_k$ constitute an orthonormal basis of $\mathbb R^d$. Set $U:=(w_1,\dots,w_{d-1})\in\mathbb R^{d\times(d-1)}$ and $\bar U:=(U,v_k)\in\mathbb R^{d\times d}$. 
By construction, $\bar U$ is an orthogonal matrix. Hence, for any $z\in\mathbb R^d$, 
\[
|z-x|^2=|\bar U^\top(z-x)|^2=|U^\top(z-x)|^2+|v_k\cdot(z-x)|^2.
\]
Therefore, if $z\in A$, then $v_k\cdot(z-x)\leq-\kappa$ and thus $|z-x|^2\geq |U^\top(z-x)|^2+\kappa^2$. Since $F_j\subset A$, we conclude
\ba{
\int_{F_j}\phi_d(z-x)\mcl H^{d-1}(dz)\leq \phi_1(\kappa)\int_{F_j}\phi_{d-1}(U^\top(z-x))\mcl H^{d-1}(dz).
}
Then, by a change of variable\footnote{For example, applying Theorem 3.9 in \cite{EvGa15} with $n=d-1,m=d,f(y)=Uy$ and $g(y)=\phi_{d-1}(y-U^\top x)1_{U^\top F_j}(y)$.}, we obtain
\ba{
\int_{F_j}\phi_d(z-x)\mcl H^{d-1}(dz)\leq \phi_1(\kappa)\int_{U^\top F_j}\phi_{d-1}(y-U^\top x)dy
\leq\phi_1(\kappa).
}

\noindent\textbf{Case 2:} $x\in A(-\kappa)$. Define the matrices $U$ and $\bar U$ as in Case 1 with $k=j$. 
Observe that we have for any $z\in F_j$
\[
\phi_{d}(z-x)=\phi_d(\bar U^\top(z-x))=\phi_{d-1}(U^\top z-U^\top x)\phi_1(b_j-x\cdot v_j).
\]
Since $x\in A(-\kappa)$, we have $b_j-x\cdot v_j\geq\kappa$. Hence 
\ba{
\int_{F_j}\phi_d(z-x)\mcl H^{d-1}(dz)\leq\phi_1(\kappa)\int_{F_j}\phi_{d-1}(U^\top(z-x))\mcl H^{d-1}(dz).
}
The remaining argument is the same as in Case 1. 

The third-derivative bound \eq{eq:vanish3} follows from a similar argument from \eq{eq:thirdderi} and we omit the details.
\end{proof}

\section{Appendix}\label{sec:app}

\subsection{Regularity of convex polytope}

\begin{lemma}\label{lem:disjoint}
Under assumptions \cref{rankassumption1,rankassumption2,rankassumption3}, we have 
\ben{\label{lem:disjoint1}
\rel(F_{j})\cap \rel(F_{j'})=\emptyset \  \text{for}\   j \neq j',
}
\ben{\label{lem:disjoint2}
\rel(F_{jk})\cap \rel(F_{j'k'})=\emptyset\  \text{for}\  \{ j,k \}\neq \{ j',k' \},
}
and
\ben{\label{lem:disjoint3}
\rel(F_{jkl})\cap \rel(F_{j'k'l'})=\emptyset\ \text{for}\  \{ j,k,l \}\neq \{ j',k',l' \},
}
where $\rel$ denotes the relative interior.
\end{lemma}
\begin{proof}
We only prove \eq{lem:disjoint2} here. The proofs of the other two statements are similar.

If $span \{ v_j,v_k \}=span \{ v_{j'},v_{k'} \}$, where $span$ denotes the subspace generated by the vectors, then
either $v_{j'}$ or $v_{k'}$ is not in $\{v_j, v_k\}$. Say $v_{j'}\notin \{v_j, v_k\}$.
Then
by assumptions \cref{rankassumption1,rankassumption2,rankassumption3}, the following equation 
\begin{equation*}
  \begin{aligned}
    v_j \cdot x=b_j,\quad  v_k \cdot x=b_k,\quad v_{j'} \cdot x=b_{j'}
  \end{aligned}
\end{equation*}
does not have any solution, so $F_{jk}$ and $F_{j'k'}$ are disjoint in this case. 

If $span \{ v_j,v_k \}\neq span \{ v_{j'},v_{k'} \}$, then $span \{ v_j,v_k \}^\perp\neq span \{ v_{j'},v_{k'} \}^\perp$, where $V^\perp$ denotes the orthogonal complement of $V$. 
Hence, there must exist a unit vector $u\in span \{ v_{j'},v_{k'} \}^\perp$ such that $v_j \cdot u>0$ or $v_k \cdot u>0$. Without loss of generality, let $v_j \cdot u>0$. Now, assume that there exists a point $a\in \rel(F_{jk})\cap \rel(F_{j'k'})$.  Then by the definition of relative interior, we can find a small number $\epsilon>0$ such that,  $a+\epsilon u \in \rel (F_{j'k'})$. Because $a \in F_{jk}$, we have $v_{j} \cdot a=b_{j}$ and $v_{j} \cdot(a+\epsilon u)>b_j $, which means that $a+\epsilon u \not\in A$. This contradicts with $a+\epsilon u \in \rel (F_{j'k'})$.
\end{proof}

\subsection{Disjointness of outer areas}

\begin{lemma}\label{lem:Sjdis}
Under assumptions \cref{rankassumption1,rankassumption2,rankassumption3}, let 
$$S_{j}:=\{y+s v_j: y\in \rel(F_{j}), s> 0\}.$$
Then, $S_{j}$ and $S_{k}$ are disjoint for any $j\neq k$. 
\end{lemma}
\begin{proof}
To obtain a contradiction, suppose that there exists a vector $x\in S_j\cap S_k$. Then we can write $x=a_j+sv_j=a_k+tv_k$ for some $a_j\in \rel(F_{j}),a_k\in \rel(F_{k})$ and $s,t>0$. 
Using the first expression of $x$, we have
\ba{
(x-a_j)\cdot(a_k-a_j)=sv_j\cdot(a_k-a_j)=s(v_j\cdot a_k-b_j)\leq 0,
}
where the second equality follows by $a_j\in\rel(F_j)\subset F_j$ and the last inequality by $a_k\in\rel(F_k)\subset A$. 
In the meantime, using the second expression of $x$, we have
\ba{
(x-a_j)\cdot(a_k-a_j)=|a_k-a_j|^2+tv_k\cdot(a_k-a_j)
=|a_k-a_j|^2+t(b_k-v_k\cdot a_j)\geq |a_k-a_j|^2,
}
where the second equality follows by $a_k\in\rel(F_k)\subset F_k$ and the last inequality by $a_j\in\rel(F_j)\subset A$. Since $\rel(F_j)\cap\rel(F_k)=\emptyset$ by \cref{lem:disjoint1}, we have $a_j\neq a_k$, so $(x-a_j)\cdot(a_k-a_j)\geq|a_k-a_j|^2>0$, a contradiction. 
\end{proof}

\begin{lemma}\label{sjkdisjoint}
Under assumptions \cref{rankassumption1,rankassumption2,rankassumption3}, for $S_{jk}$ defined in \cref{eq:Sjk}, we have $S_{jk}\cap S_{j'k'}=\emptyset$, for any $\{ j,k \}\neq \{ j',k' \}$.
\end{lemma}
\begin{proof}
To obtain a contradiction, suppose that there exists a vector $x\in S_{jk}\cap S_{j'k'}$. Then we can write $x=a_{jk}+sv_j+tv_k=a_{j'k'}+s'v_{j'}+t'v_{k'}$ for some $a_{jk}\in \rel(F_{jk}),a_{j'k'}\in \rel(F_{j'k'})$ and $s,t,s',t'>0$. 
Using the first expression of $x$, we have
\ba{
(x-a_{jk})\cdot(a_{j'k'}-a_{jk})
&=(sv_j+tv_k)\cdot(a_{j'k'}-a_{jk})\\
&=s(v_j\cdot a_{j'k'}-b_j)+t(v_k\cdot a_{j'k'}-b_k)\leq 0,
}
where the second equality follows by $a_{jk}\in\rel(F_{jk})\subset F_{jk}$ and the last inequality by $a_{j'k'}\in\rel(F_{j'k'})\subset A$. 
In the meantime, using the second expression of $x$, we have
\ba{
(x-a_{jk})\cdot(a_{j'k'}-a_{jk})
&=|a_{j'k'}-a_{jk}|^2+(s'v_{j'}+t'v_{k'})\cdot(a_{j'k'}-a_{jk})\\
&=|a_{j'k'}-a_{jk}|^2+s'(b_{j'}-v_{j'}\cdot a_{jk})+t'(b_{k'}-v_{k'}\cdot a_{jk})\\
&\geq |a_{j'k'}-a_{jk}|^2,
}
where the second equality follows by $a_{j'k'}\in\rel(F_{j'k'})\subset F_{j'k'}$ and the last inequality by $a_{jk}\in\rel(F_{jk})\subset A$. Since $\rel(F_{jk})\cap\rel(F_{j'k'})=\emptyset$ by \cref{lem:disjoint2}, we have $a_{jk}\neq a_{j'k'}$, so $(x-a_{jk})\cdot(a_{j'k'}-a_{jk})\geq|a_{j'k'}-a_{jk}|^2>0$, a contradiction. 
\end{proof}

\begin{lemma}
    \label{sjkldisjoint}
    Under assumptions \cref{rankassumption1,rankassumption2,rankassumption3}, for $S_{jkl}$ defined in \cref{eq:Sjkl}, we have $S_{jkl}\cap S_{j'k'l'}=\emptyset$, for any $\{ j,k,l \}\neq \{ j',k' ,l'\}$. 
\end{lemma}
The proof of \cref{sjkldisjoint} is similar to that of \cref{sjkdisjoint} and is omitted.

\subsection{Proof of (\ref{eq:Fjkl2})}\label{sec:Fjkl2}

The case $\text{dist}(0, F_{jkl})>\sqrt{6 \log d}$ can be dealt with easily using the union bound as in \eq{eq:R1jtruncation} (we changed the threshold because there are $O(d^3)$ terms in the summation now). 
In the following, we consider the case $\text{dist}(0, F_{jkl})\leq\sqrt{6\log d}$.

Denote by $\bar x_{jkl}$ the projection of 0 to the $(d-3)$-dimensional affine subspace containing $F_{jkl}$. Assume $\bar x_{jkl}\ne 0$ for ease of discussion. Otherwise, the desired result follows from a similar and simpler argument. We now define a set $\bar S_{jkl}$ by decreasing and (possibly) rotating $S_{jkl}$ around $\bar x_{jkl}$ in the three-dimensional subspace spanned by $v_j, v_k$ and $v_l$ as follows. Consider the two-dimensional sphere centered at $\bar x_{jkl}$ with unit radius in $span\{v_j, v_k, v_l\}$. Extend the line from 0 to $\bar x_{jkl}$ until it intersects with the sphere, call the intersection the (north) pole $\bf p$. The three points ${\bf b}:=\bar x_{jkl}+v_j$, ${\bf c}:=\bar x_{jkl}+v_k$ and ${\bf d}:=\bar x_{jkl}+v_l$ form a spherical triangle. 
Find the longest edge of this triangle, say $\bf bc$ (we use $\bf bc$ to denote both the geodesic from $\bf b$ to $\bf c$ and its length). 
We assume $\bf d$ has a unique geodesic projection point $\bf d'$ to the great circle containing $\bf b$ and $\bf c$. Otherwise, we perturb $\bf d$ by an infinitesimal amount and use an approximation argument. Note that ${\bf dd'}<\pi/2$.
We first consider the case that $\bf d'$ is on the geodesic from the midpoint of $\bf bc$ to ${\bf \hat c}$ (excluding ${\bf \hat c}$), where ${\bf \hat c}$ is the opposite point of $\bf c$ on the sphere.
Let $\bf e$ be the midpoint of $\bf bc$ and let $\bf f$ be on $\bf ec$ such that ${\bf ef}=\alpha/100$ (this is possible because ${\bf bc}\geq \alpha$). Move $\bf e$ and $\bf f$ towards the inside of the triangle in spherically orthogonal directions to $\bf bc$ until they intersect with $\bf dc$ at $\bf e'$ and $\bf f'$, respectively. If ${\bf bc}>\pi/10$, then ${\bf d'c}<20{\bf ec}$. From Napier's rules for right spherical triangles, we have
\ben{\label{eq:napier}
\tan({\bf ee'})/\tan({\bf dd'})=\sin({\bf ec})/\sin({\bf d'c}).
} 
Note that $0<{\bf ec}<{\bf d'c}<\pi$. If ${\bf ee'}\geq {\bf dd'}$, then ${\bf ee'}/{\bf dd'}>{\bf ec}/{\bf d'c}$.
If ${\bf ee'}< {\bf dd'}$, then from \eq{eq:napier}, the convexity of $\tan$ on $(0,\pi/2)$ and the concavity of $\sin$ on $(0,\pi)$, we again have ${\bf ee'}/{\bf dd'}>{\bf ec}/{\bf d'c}$. This, together with the minimum angle assumption ${\bf dd'}\geq \beta$, implies ${\bf ee'}>\beta/100$. Similarly, we have ${\bf ff'}>\beta/100$. Therefore, we can find a spherical quadrilateral $\bf efhg$ inside of the triangle such that ${\bf ef}=\alpha/100$, ${\bf eg}={\bf fh}=\beta/100$ and the spherical angles at $\bf e$ and $\bf f$ are both equal to $\pi/2$. If ${\bf bc}\leq \pi/10$, then $d'$ must be on $\bf bc$ because $\bf bc$ is the longest edge of the triangle. Then, following the same argument as above, we can again find the above spherical quadrilateral $\bf efhg$.
Next, the case that $\bf d'$ is on the geodesic from the midpoint of $\bf bc$ to $\bf \hat b$ (excluding $\bf \hat b$), where $\bf \hat b$ is the opposite point of $\bf b$ on the sphere, can be considered similarly. Finally, for the case that $\bf d'$ is on the geodesic from $\bf \hat c$ to $\bf \hat b$, we can first decrease the triangle by rotating $\bf d$ towards $\bf bc$ such that its projection $\bf d'$ is on $\bf bc$ and then follow the arguments for the previous two cases. 
Let the four vertices of the spherical quadrilateral found above be $\bar x_{jkl}+v_{i,jkl},\  i=1,\dots, 4$.
If the spherical quadrilateral formed by these four points contains the pole, then do nothing. Otherwise, find the geodesic projection $a_{jkl}$ of the pole to this spherical quadrilateral (if it is not unique, pick anyone) and rotate this spherical rectangle along the geodesic from $a_{jkl}$ to the pole. 

We claim that the above rotation moves each point in the spherical quadrilateral closer to the pole. It is obviously true if the geodesic quadrilateral is not a subset of the northern hemisphere (recall the quadrilateral is sufficiently small).
If the spherical quadrilateral is a subset of northern hemisphere, we consider the spherical triangle formed by the pole $\bf p$, ${\bf a}:=a_{jkl}$, and any point $\bf q$ in the spherical quadrilateral. Let $\angle {\bf paq}$ be the (spherical) angle of the vertex $\bf a$. We have $\angle {\bf paq}\geq \pi/2$; otherwise, there is a point on the geodesic from $\bf a$ to $\bf q$ that is closer to $\bf p$, a contradiction. Let $0< {\bf pa}, {\bf qa}, {\bf pq}<\pi/2$ be the (spherical) edge length of the triangle. From the spherical law of cosines, we have
\be{
\cos({\bf pq})=\cos({\bf qa})\cos({\bf pa})+\sin({\bf pa})\sin({\bf qa})\cos(\angle {\bf paq})\leq \cos({\bf qa})\cos({\bf pa}).
}
This implies $\cos({\bf pq})<\cos({\bf qa})$, and hence ${\bf pq}>{\bf qa}$. This proves the claim.

Let $\bar x_{jkl}+\bar v_{i,jkl},\  i=1,\dots, 4$. be the four vertices of the spherical quadrilateral after above rotation. Define
\be{
\bar S_{jkl}:=\{y+s_1 \bar v_{1,jkl}+s_2 \bar v_{2,jkl}+s_3 \bar v_{3,jkl}+s_4 \bar v_{4,jkl}: y\in relint(F_{jkl}), s_1, s_2, s_3, s_4>0\}.
}
Because the above rotation moves each point in the spherical quadrilateral closer to the pole, it increases their Euclidean norms.
Therefore,
\ben{\label{eq:gammasjkl}
\gamma_d(\bar S_{jkl})\leq \gamma_d(S_{jkl}).
}
Let $u_{jkl}=\bar x_{jkl}/|\bar x_{jkl}|$ (recall we have assumed $\bar x_{jkl}\ne 0$).
Let $w_1, w_2, u_{jkl}$ be an orthonormal basis of $span\{v_j, v_k, v_l\}$ and
let 
\be{
u_{1,jkl}=w_1+\frac{1}{\sqrt{2}}u_{jkl},\ u_{2,jkl}=-\frac{1}{2}w_1+\frac{\sqrt{3}}{2}w_2+\frac{1}{\sqrt{2}}u_{jkl},\ u_{3,jkl}=-\frac{1}{2}w_1-\frac{\sqrt{3}}{2}w_2+\frac{1}{\sqrt{2}}u_{jkl}.
}
Note that the pair-wise angles of the above three (unnormalized) vectors are all equal to $\pi/2$.
Similar to \eq{eq:star1} and \eq{eq:NFjk}, we bound three standard normal densities by normal tail probabilities in three orthogonal directions to $F_{jkl}$ and obtain (recall we have assumed $\text{dist}(0, F_{jkl})\leq\sqrt{6\log d}$)
\ben{\label{Fjkl}
  \left| \int_{F_{jkl}} \phi_d(z) \mcl{H}^{d-3}(d z)\right|\leq (\text{dist}(0, F_{jkl})+1)^3 \gamma_d(N_{F_{jkl}})\leq  C (\log d)^{3/2} \gamma_d(N_{F_{jkl}}),
}
where
\ben{\label{eq:NFjkl}
N_{F_{jkl}}:=\{y+r u_{1,jkl}+s u_{2,jkl}+t u_{3,jkl}: y\in \rel(F_{jkl}), r>0, s>0, t>0\}.
}
Computing standard Gaussian probabilities by first integrating along the coordinate $u_{jkl}$, we have, with $t=|\bar x_{jkl}|$,
\be{
\frac{\gamma_d(N_{F_{jkl}})}{\gamma_d(\bar S_{jkl})}=\frac{\int_t^\infty \phi_1(u) \gamma_2[(u-t)A] du}{\int_t^\infty \phi_1(u) \gamma_2[(u-t)B] du},
}
where $A$ is an equilateral triangle with edge length $\sqrt{6}$ and centered at 0 and $B$ is a quadrilateral mapped from a small spherical quadrilateral to the tangent plane at 0, which is inside or on the boundary of the spherical quadrilateral. By construction, this small spherical quadrilateral has edge lengths proportional to $\alpha, \beta, \alpha, \beta$ and four angles all sufficiently close to $\pi/2$. This implies $B$ also has edge lengths proportional to $\alpha, \beta, \alpha, \beta$ and four angles all sufficiently close to $\pi/2$. Because Gaussian density is decreasing in the Euclidean norm of its arguments, we have $\gamma_2[(u-t)A]/\gamma_2[(u-t)B]\leq C/\alpha\beta$ and
\ben{\label{eq:star2}
\gamma_d(N_{F_{jkl}})\leq \frac{C}{\alpha\beta} \gamma_d(\bar S_{jkl})\leq \frac{C}{\alpha\beta} \gamma_d( S_{jkl}),
}
where we used \eq{eq:gammasjkl} in the last inequality.

 The desired bound \eq{eq:Fjkl2} follows from \eq{Fjkl} and \eq{eq:star2}.




\section*{Acknowledgements}

Fang X. was partially supported by Hong Kong RGC GRF 14305821 and 14304822 and a CUHK direct grant. Koike Y. was partially supported by JST CREST Grant Number JPMJCR2115 and JSPS KAKENHI Grant Numbers JP19K13668, JP22H00834, JP22H00889, ZJP22H01139. 

\bibliographystyle{apalike}
\bibliography{reference}

\end{document}